\documentclass{article}

\usepackage[english]{babel}

\usepackage[a4paper,top=2.5cm,bottom=2.5cm,left=2.5cm,right=2.5cm,marginparwidth=1.75cm]{geometry}

\usepackage{subcaption}
\usepackage{amsmath,amsthm,amsfonts,amssymb}
\usepackage[hidelinks]{hyperref}

\usepackage{algorithm}%
\usepackage{algorithmicx}%
\usepackage{algpseudocode}%

\newtheorem{theorem}{Theorem}

\newtheorem{corollary}{Corollary}
\newtheorem{proposition}{Proposition}
\newtheorem{definition}{Definition}
\newtheorem{example}{Example}
\newtheorem{remark}{Remark}

\usepackage{stmaryrd} 
\newcommand{\iverson}[1]{\llbracket#1\rrbracket}

\usepackage{mathtools}
\DeclarePairedDelimiterX\ourset[1]\lbrace\rbrace{\def\mid{\;\delimsize\vert\;}\,#1\,}
\DeclarePairedDelimiterX{\ourabs}[1]\lvert\rvert{#1}

\usepackage{calc}
\newsavebox\CBox
\newcommand\strike[2][0.5pt]{%
  \ifmmode\sbox\CBox{$#2$}\else\sbox\CBox{#2}\fi%
  \makebox[0pt][l]{\usebox\CBox}%
  {\color{red}\rule[0.5\ht\CBox-#1/2]{\wd\CBox}{#1}}}

\usepackage{pdflscape}
\usepackage{afterpage}

\DeclareMathOperator*{\argmin}{argmin}
\DeclareMathOperator*{\argmax}{argmax}

\DeclareMathOperator*{\conv}{conv}
\DeclareMathOperator*{\cone}{cone}

\newcommand*{\UB}{B}
\DeclareMathOperator*{\SOL}{SOL}
\DeclareMathOperator*{\OPT}{OPT}
\newcommand*{\minimal}{A_{\min}}
\newcommand*{\domain}{D}
\newcommand*{\TPA}[1]{T_{#1}} 
\newcommand*{\VAL}[2]{\<{#1},\phi({#2})\>} 
\newcommand*{\vars}{V}
\newcommand*{\constraints}{C}
\newcommand*\scope{S}
\newcommand*{\deactivating}{R}
\newcommand*{\chosenindices}{I} 
\newcommand*{\edges}{E}
\newcommand*{\CSP}{A}
\newcommand*{\Mset}{\deactivating^*} 
\DeclareMathOperator*{\COL}{COL}

\newcommand*{\R}{\mathbb{R}}

\def\<{\langle}
\def\>{\rangle}
\newcommand*{\val}[1]{\mathsf{#1}}

\renewcommand*{\S}{§}

\title{Super-Reparametrizations of Weighted CSPs: \\Properties and Optimization Perspective}
\author{Tom\'{a}\v{s} Dlask$^{1,}$\footnote{Corresponding author.}\;\,, Tom\'{a}\v{s} Werner$^1$, and Simon de Givry$^2$}
\date{$^1$ Department of Cybernetics, Faculty of Electrical Engineering, Czech Technical University in Prague, Karlovo n\'{a}m\v{e}st\'{i} 13, Prague, 12000, Czech Republic\\ [1ex]
      $^2$ Université Fédérale de Toulouse, ANITI, INRAE, UR 875, 31320 Castanet-Tolosan, France\\ [2ex]
      {\normalsize \texttt{dlaskto2@fel.cvut.cz, werner@fel.cvut.cz, simon.de-givry@inrae.fr}}}

\usepackage[absolute,overlay]{textpos}

\begin{document}

\begin{textblock*}{21cm}(0cm,0.3cm)
\centering
The version of record of this article is published in Constraints and is available online at \url{https://doi.org/10.1007/s10601-023-09343-6}.
\end{textblock*}

\emergencystretch 3em

\renewcommand*{\thefootnote}{\fnsymbol{footnote}}

\maketitle

\renewcommand*{\thefootnote}{\arabic{footnote}}

\begin{abstract}\noindent
The notion of reparametrizations of Weighted CSPs (WCSPs) (also known as equivalence-preserving transformations of WCSPs) is well-known and finds its use in many algorithms to approximate or bound the optimal WCSP value. In contrast, the concept of super-reparametrizations (which are changes of the weights that keep or increase the WCSP objective for every assignment) was already proposed but never studied in detail. To fill this gap, we present a number of theoretical properties of super-reparametrizations and compare them to those of reparametrizations. Furthermore, we propose a framework for computing upper bounds on the optimal value of the (maximization version of) WCSP using super-reparametrizations. We show that it is in principle possible to employ arbitrary (under some technical conditions) constraint propagation rules to improve the bound. For arc consistency in particular, the method reduces to the known Virtual AC (VAC) algorithm. We implemented the method for singleton arc consistency (SAC) and compared it to other strong local consistencies in WCSPs on a public benchmark. The results show that the bounds obtained from SAC are superior for many instance groups.
\end{abstract}


\section{Introduction}\label{se:introduction}

In the \emph{weighted constraint satisfaction problem (WCSP)\/} we maximize the sum of (weight) functions over many discrete variables, where each function depends only on a (usually small) subset of the variables. A popular approach to tackle this NP-hard combinatorial optimization problem is via its linear programming (LP) relaxation \cite{schlesinger1976,Werner-PAMI07,Wainwright08,Zivny-VCSPbook-2012,Bogdan-book-2019}. The dual of this LP relaxation \cite{Werner-PAMI07,Bogdan-book-2019,Cooper-AI-2010} can be interpreted as follows. Feasible dual solutions correspond to \emph{reparametrizations} (also known as equivalence-preserving transformations~\cite{Cooper-AI-2010}) of the WCSP objective function, which are obtained by moving weights between weight functions so that the WCSP objective function is preserved. The dual LP relaxation then seeks to find such a reparametrization of the initial WCSP that minimizes an \emph{upper bound} on the WCSP objective value by reparametrizations. For some instances, the minimal upper bound is equal to the maximal value of the WCSP objective (i.e., the LP relaxation is tight) but, in general, there is a gap between them. The precise form of the dual LP differs slightly from author to author.

For larger instances, solving the LP relaxation to global optimality is too costly. Therefore, the upper bound is usually minimized suboptimally by performing reparame\-trizations only locally. Stopping points of these suboptimal methods are usually characterized by various levels of local consistency of the CSP formed by the \emph{active tuples} (i.e., the tuples with the maximum weight in each weight function individually) of the reparametrized WCSP. This is consistent with the fact that a necessary (but not sufficient) condition for global optimality of the dual LP relaxation is that the active-tuple CSP has a non-empty local consistency closure. The level of local consistency at optimum depends on the space of allowed reparametrizations: if weights can move only between pairs of weight functions of which one is unary or nullary, it is arc consistency (AC); if weights can move between two weight functions of any arity, it is pairwise consistency (PWC). These suboptimal methods can be divided into two main classes.

The first class, popular in computer vision and machine learning, is known as \emph{convex message passing} \cite{trws,mplp,tourani2018mplp++,tourani2020taxonomy,Werner-PAMI07,Werner-PAMI-2010,Kolmogorov-SRMP}. These methods repeat a simple local operation and can be seen as block-coordinate descent with exact updates satisfying the so-called relative interior rule~\cite{cvpr_werner2020}. At fixed points, the active-tuple CSP has non-empty AC (or PWC) consistency closure. These methods yield good upper bounds but are too slow to be applied in each node of a branch-and-bound search.

The second class has been called \emph{soft local consistency} methods in constraint programming~\cite{Cooper-AI-2010}, due to its similarity to local consistencies in the ordinary CSP. One type of these methods moves only integer weights between weight functions (assuming all initial weights are integer) and is efficient enough to be maintained during search. Its most advanced representant is \emph{existential directional arc consistency (EDAC)} algorithm~\cite{de2005existential}. The other type allows moving fractional weights, which can lead to better bounds but is more costly, hence usually not suitable to be applied during search. Its representants are the \emph{virtual arc consistency (VAC)} algorithm~\cite{cooper2008virtual,Cooper-AI-2010} and the very similar Augmenting DAG algorithm \cite{Koval76,Werner-PAMI07,Werner-TR-2005}. These methods are based on the following fact: whenever the active-tuple CSP has an empty AC closure, there exists a reparametrization of the WCSP that decreases the upper bound. 
Thus, each iteration of these algorithms first applies the AC algorithm to the active-tuple CSP and if domain wipe-out occurs, it constructs a dual-improving direction by back-tracking the history of the AC algorithm, and finally reparametrizes the current WCSP by moving along this direction by a suitable step size.
The VAC and Augmenting DAG algorithms converge to a non-unique state when the active-tuple CSP has a non-empty AC closure (which is called virtual arc consistency) but are typically faster than convex message passing methods.

In the soft-consistency terminology, global optima of the dual LP relaxation have been called \emph{optimally soft arc consistent (OSAC)} WCSPs \cite{cooper2007optimal,Cooper-AI-2010}. In this sense, EDAC and VAC are relaxations of OSAC. But note that OSAC can no longer be considered a local consistency, since no algorithm using only local operations is known to enforce it\footnote{And it is unlikely that such an algorithm exists, since it has been proved \cite{Prusa-PAMI-2015,Prusa-Werner-SIAM2019} that finding global optimum of the LP relaxation of the WCSP is not easier than solving the general linear programming problem.}.

Reparametrizations in general cannot enforce stronger local consistencies of the active-tuple CSP than PWC. This can be seen as follows: if the active-tuple CSP has a non-empty PWC closure but violates some stronger local consistency (hence it is unsatisfiable), there exists no reparametrization that would decrease the upper bound and possibly make the active-tuple CSP satisfy the stronger local consistency. The only way to achieve stronger local consistencies (such as $k$-consistencies) of the active-tuple CSP by reparametrizations is to introduce new weight functions (of possibly higher arities) and then move weights between these new weight functions and the existing weight functions. This allows constructing a hierarchy of progressively tighter LP relaxations of the WCSP \cite{Werner-PAMI-2010,sontag2008tightening,nguyen2017triangle,batra2011tighter}, including the Sherali-Adams hierarchy \cite{Thapper-ICALP-2015}.

In this paper, we study a different LP-based approach, namely an LP formulation of the WCSP, which was proposed in~\cite{komodakis2008beyond} but never pursued later. It differs from the above well-known LP relaxation and does not belong to the hierarchy of LP relaxations obtained by introducing new weight functions of higher arities. This LP formulation minimizes the same upper bound on the WCSP objective value but this time over \emph{super-reparametrizations} of the initial WCSP objective function, which are changes of the weights that either preserve or increase the WCSP objective value for every assignment. This LP formulation has an exponential number of inequality constraints (representing super-reparametrizations) and is exact, i.e., its minimal value is always equal to the maximal value of the WCSP objective.

We propose to solve this LP suboptimally by a local search method, which is based on the following key observation: whenever the active-tuple CSP is unsatisfiable, there exists a super-reparametrization (but possibly no reparametrization) that decreases the upper bound. The direction of this super-reparametrization is a \emph{certificate of unsatisfiability} of the active-tuple CSP, which can be constructed from the history of the CSP solver. Note that this approach strictly generalizes the VAC algorithm: if the active-tuple CSP has a non-empty AC closure but is unsatisfiable, the VAC algorithm is stuck (because no reparametrization can decrease the upper bound) but our algorithm can decrease the bound by a super-reparametrization. The cost for this greater generality is that super-reparametrizations may preserve neither the WCSP objective value for some assignments nor the set of optimal assignments, but they can nevertheless provide valid, and possibly tighter, upper bounds on the WCSP optimal value.

After formulating this general framework, we focus on the case when the unsatisfiability of the active-tuple CSP is proved by local consistencies stronger than AC/PWC. In particular we use \emph{singleton arc consistency (SAC)}, which is interesting because it does not have bounded support~\cite{bessiere2008theoretical} and therefore it would be difficult to achieve by introducing new weight functions of higher arity. We show how to construct a certificate of unsatisfiability of a CSP from the history of the SAC algorithm. Our algorithm then interleaves AC and SAC: we always keep decreasing the upper bound by reparametrizations until the active-tuple CSP has non-empty AC closure, and only then decrease the bound by a super-reparametrization if the SAC closure of active-tuple CSP is found empty. In experiments we show that on many WCSP instances, this algorithm yields better bounds than state-of-the-art soft local consistency methods in reasonable runtime. Note, we report only the achieved upper bounds but do not use them in branch-and-bound search, which would be beyond the scope of our paper.

To the best of our knowledge, super-reparametrizations have not been utilized or studied except for~\cite{komodakis2008beyond} and~\cite{sontag2009tree}. In~\cite{komodakis2008beyond}, super-reparametrizations were used to obtain tighter bounds using a specialized cycle-repairing algorithm and were identified in~\cite{sontag2009tree} as a property satisfied by all formulations of the linear programming relaxations based on reparametrizations. However, \cite{sontag2009tree}~focuses almost solely on the relation between different formulations of reparametrizations, instead of super-reparametrizations. To fill in this gap, we theoretically analyze the associated optimization problem and also the properties of super-reparametrizations.

Compared to the previous version of this paper~\cite{cp2021}, we improved the current paper in the following ways:
\begin{itemize}
    \item Most importantly, we include a study on the theoretical properties of super-reparametrizations and compare them to those of reparametrizations (\S\ref{se:theoretical_properties}).
    \item Although our implementation remains limited to WCSPs of arity 2, we present all of our theoretical results for WCSPs of any arity.
    \item We include a geometric interpretation that provides intuitive insights and thus simplifies understanding of our method (\S\ref{se:properties_of_the_method}).
    \item In addition to making our code publicly available, we add more information on implementation details to improve reproducibility (\S\ref{se:final_alg}).
    \item We also analyze the cone of non-negative weighted CSPs and prove that it is dual to the marginal polytope (\S\ref{sec:super-reparam}).
    \item Unsurprisingly, we show that some decision problems connected to our approach and super-reparametrizations are NP-hard (in~\S\ref{se:hardness}).
\end{itemize}

\paragraph*{Structure}

We begin in~\S\ref{se:notation} by formally defining the Weighted CSP, classical (crisp) CSP, and introducing the notation that will be used throughout the paper. Then, in~\S\ref{se:bounding}, we formally define the optimization problem of minimizing an upper bound over reparametrizations and/or super-reparametrizations where we also state the sufficient and necessary optimality conditions. Next, \S\ref{se:optimizing_bound} proposes a practical approach for approximate minimization of the upper bound over super-reparametrizations using constraint propagation. We also give experimental results comparing our approach with existing soft local consistencies. Additional properties of the underlying active-tuple CSPs (see definition later) and the sets of optimal (or also non-optimal) super-reparametrizations are given in~\S\ref{se:theoretical_properties}. \S\ref{se:hardness}~presents the hardness results. We provide a detailed example demonstrating EDAC, VAC, and our proposed approach with SAC in Appendix~\ref{ap:example}.

\section{Notation}\label{se:notation}

Let $\vars$~be a finite set of \emph{variables\/} and $\domain$~a finite \emph{domain\/} of each variable. An \emph{assignment\/} $x\in \domain^\vars$ assigns\footnote{As usual, $\domain^\vars$ denotes the set of all mappings from $\vars$ to $\domain$, so $x\in\domain^\vars$ is the same as $x\colon \vars\to \domain$.} a value $x_i\in \domain$ to each variable $i\in \vars$. Let $\constraints \subseteq 2^\vars$ be a set of non-empty \emph{scopes}, i.e., $(\vars,\constraints)$ can be seen as an undirected hypergraph. The triplet $(\domain,\vars,\constraints)$ defines the \emph{structure} of a (weighted) CSP and will be fixed throughout the paper. By
\begin{equation}
T = \ourset{(\scope,k)\mid \scope\in \constraints, \; k \in \domain^\scope}
= \bigcup_{\scope\in \constraints} \TPA{\scope} \quad\text{where}\quad \TPA{\scope}=\ourset{(\scope,k) \mid k \in \domain^\scope}
\label{eq:tuples}
\end{equation}
we denote the set of \emph{tuples\/}, partitioned into sets $\TPA{\scope}$, $\scope\in\constraints$. We say that an assignment $x\in\domain^\vars$ \emph{uses} a tuple $t=(S,k)\in T$ if $x[S]=k$ where  $x[\scope]$ denotes the restriction of~$x$ onto the set $\scope\subseteq \vars$, i.e., for $\scope=\{i_1, ..., i_{|\scope|}\}$ we have  $x[\scope]=(x_{i_1}, ..., x_{i_{|\scope|}})$ (where the order of the components is defined by the total order on~$\scope$ inherited from some arbitrary fixed total order on~$\vars$). Each assignment~$x\in\domain^\vars$ uses exactly one tuple from each~$\TPA{\scope}$.

An instance of the \emph{constraint satisfaction problem (CSP)} is defined by the quadruple $(\domain,\vars,\constraints,\CSP)$ where $\CSP\subseteq T$ is the set of \emph{allowed tuples} (while the tuples $T-\CSP$ are \emph{forbidden}). As the CSP structure $(\domain,\vars,\constraints)$ will be always the same, we will refer to the CSP instance only as~$\CSP$ (in other words, in the sequel we identify CSP instances with subsets of~$T$). An assignment $x\in\domain^\vars$ is a \emph{solution} to a CSP $\CSP\subseteq T$ if it uses only allowed tuples, i.e., $(\scope,x[\scope])\in\CSP$ for all $\scope\in \constraints$.
The set of all solutions to the CSP will be denoted by $\SOL(\CSP)\subseteq D^V$. The CSP is \emph{satisfiable} if $\SOL(\CSP)\neq\emptyset$, otherwise it is \emph{unsatisfiable}.

The \emph{weighted constraint satisfaction problem (WCSP)\/}\footnote{The WCSP is also known under different names, e.g., as the finite-valued CSP \cite{thapper2016complexity,kolmogorov2015power}, discrete energy minimization~\cite{Kappes-study-2015}, or maximum a posteriori (MAP) inference in graphical models \cite{Bogdan-book-2019}. It is also the main task in cost function networks~\cite{cooper2020valued}.} seeks to find an assignment $x\in \domain^\vars$ that maximizes the function 
\begin{equation}
\sum_{\scope\in\constraints} f_\scope(x[\scope])
\label{eq:energy}
\end{equation}
where $f_\scope\colon\domain^\scope\to\R$, $\scope\in\constraints$, are given \emph{weight functions}. All the weights (i.e., the values of the weight functions) together can be seen as a vector $f\in\R^T$, such that for $t=(\scope,k)\in T$ we have $f_t=f_\scope(k)$. The WCSP instance is defined by the quadruple $(\domain,\vars,\constraints,f)$. However, as the structure $(\domain,\vars,\constraints)$ will be always the same, we will refer to WCSP instances only as~$f$ (in other words, we identify WCSP instances with vectors from~$\R^T$).

\begin{example}
For example, if $\vars = \{1,2,3,4\}$, $\constraints=\{\{1\},\{2\},\{2,3\},\{1,4\},\{2,3,4\}\}$, and $\domain=\{\val{a,b}\}$, then we want to maximize the expression
\begin{equation*}
    f_{\{1\}}(x_1)+f_{\{2\}}(x_2)+f_{\{2,3\}}(x_2,x_3)+f_{\{1,4\}}(x_1,x_4)+f_{\{2,3,4\}}(x_2,x_3,x_4)
\end{equation*}
over $x_1,x_2,x_3,x_4\in \{\val{a,b}\}$. We have, e.g.,
\begin{equation*}
    \TPA{{\{2,3\}}} = \{(\{2,3\},(\val{a,a})), (\{2,3\},(\val{a,b})), (\{2,3\},(\val{b,a})), (\{2,3\},(\val{b,b}))\}.
\end{equation*}
\end{example}

\begin{remark}
In some formalisms~\cite{Cooper-AI-2010,nguyen2017triangle}, the objective~\eqref{eq:energy} is to be minimized. For our purposes, these settings are equivalent and the results for minimization problems are analogous as one can invert the sign of all weights and maximize instead. Next, some papers consider only non-negative weights and the empty (nullary) scope~$\emptyset\in\constraints$ whose weight~$f_\emptyset$ constitutes a bound on the WCSP optimal value~\cite{Cooper-AI-2010,nguyen2017triangle}. However, we will later need both positive and negative weights in a WCSP, so we require~$\emptyset\notin\constraints$ to simplify notations (also, with both positive and negative weights, $f_\emptyset$ would not yield a bound on the optimal value).
\end{remark}

We will use another notation for the WCSP objective, which is common in machine learning, see, e.g.,~\cite[\S3]{Wainwright08}. We define an indicator map $\phi\colon D^V\to\{0,1\}^T$ by
\begin{equation}
\phi_t(x) = \iverson{x[\scope]=k} \quad \text{ for each } t=(S,k)\in T
\label{eq:phi}
\end{equation}
where $\iverson{\cdot}$ denotes the \emph{Iverson bracket}, which equals~1 if the logical expression in the bracket is true and~0 if it is false. The WCSP objective~\eqref{eq:energy} can now be written as the dot product
\begin{equation}
\sum_{\scope\in\constraints} f_\scope(x[\scope])
= \sum_{t\in T} f_t\phi_t(x)
= \langle f,\phi(x)\rangle .
\label{eq:energy'}
\end{equation}
This makes explicit that the WCSP objective is linear in the weight vector~$f$. The WCSP optimal value is
\begin{equation}\label{eq:optimization_task}
\max_{x\in \domain^\vars} \VAL fx
= \max_{\mu\in M} \<f,\mu\>
\end{equation}
where
\begin{equation}\label{eq:M}
M = \phi(\domain^\vars) = \ourset{\phi(x)\mid x\in D^V} \subseteq \{0,1\}^T .
\end{equation}
Note that $M$~is defined only by the structure $(D,V,C)$.

\section{Bounding the WCSP Optimal Value}\label{se:bounding}

We define the function $\UB\colon \R^T\to\R$ by
\begin{equation}\label{eq:def_ub}
\UB(f) = \sum_{\scope\in \constraints} \max_{k \in \domain^\scope}f_\scope(k) = \sum_{\scope \in \constraints}\max_{t \in \TPA{\scope}} f_t.
\end{equation}
This is a convex piecewise-affine function. For $f\in\R^T$, we call a tuple $t=(S,k)\in T$ \emph{active\/}\footnote{\label{foot:toLP} Our term `active tuple' comes from the term `active inequality'. Indeed, \eqref{eq:def_ub} can be calculated as the minimum of $\sum_{\scope\in \constraints}z_S$ subject to $z_S\ge f_t\;\forall t\in \TPA{\scope}$, where $z_S\in\R$ are auxiliary variables. At optimum, we have $z_S=\max_{t \in \TPA{\scope}} f_t$ and an inequality $z_S\ge f_t$ is active if and only if tuple~$t$ is active.} if
\begin{equation}\label{eq:active}
f_t = \max_{t'\in \TPA{\scope}} f_{t'} .
\end{equation}
The set of all tuples that are active for~$f$ is denoted\footnote{The set $A^*(f)$ corresponds to the notion of~Bool($f$) in~\cite{Cooper-AI-2010}. The characteristic vector of the set $A^*(f)$ was denoted~$\bar f$ in~\cite{tourani2018mplp++,Werner-PAMI07}, $\lceil f\rceil$ in~\cite{Werner-PAMI-2010}, and mi$[f]$ in~\cite{Bogdan-book-2019}.} by $A^*(f)\subseteq T$. Note, $A^*(f)\subseteq T$ can be interpreted as a CSP.

\begin{theorem}[\cite{Werner-PAMI07}]\label{th:ub}
For every WCSP $f\in\R^T$ and every assignment $x\in \domain^\vars$ we have:
\begin{enumerate}
\item[(a)] $\UB(f)\ge \VAL fx$,
\item[(b)] $\UB(f)=\VAL fx$ if and only if $x\in\SOL(A^*(f))$.
\end{enumerate}
\end{theorem}
\begin{proof}
Statement (a) can be checked by comparing expressions~\eqref{eq:energy} and~\eqref{eq:def_ub} term by term.

Statement (b) says that $\UB(f)=\VAL fx$ if and only if $(\scope,x[\scope])\in A^*(f)$ for all~$\scope\in\constraints$. This is again straightforward from~\eqref{eq:energy} and~\eqref{eq:def_ub}.
\end{proof}

\noindent
Theorem~\ref{th:ub} says that $B(f)$ is an \emph{upper bound\/} on the WCSP optimal value. Moreover, it shows that $\UB(f)=\VAL fx$ implies that $x$~is a maximizer of the WCSP objective~\eqref{eq:energy}.

\begin{example}\label{ex:example_notation}
Let~$\vars=\{1,2\}$, $\domain=\{\val a,\val b\}$, and $\constraints=\{\{1\},\{2\},\{1,2\}\}$. For this structure, the set of tuples is
\begin{align}\label{eq:tuples_example}
\begin{split}
    T= \{&(\{1\},\val a),(\{1\},\val b),(\{2\},\val a),(\{2\},\val b),\\
    &(\{1,2\},(\val a,\val a)),(\{1,2\},(\val a,\val b)),(\{1,2\},(\val b,\val a)),(\{1,2\},(\val b,\val b))\}.
\end{split}
\end{align}
For assignment $x=(\val a,\val b)\in\domain^\vars$ (i.e., $x_1=\val a$, $x_2=\val b$), we have $\phi(x) = (1,0,0,1,0,1,0,0)\in\{0,1\}^T$ where the order of the tuples is given by~\eqref{eq:tuples_example}. 

An example of a WCSP~$f$ with this structure is shown in Figure~\ref{fig:example_notation}. The set of tuples active for~$f$ is
\begin{equation}
    A^*(f)= \{(\{1\},\val b),(\{2\},\val a),(\{1,2\},(\val b,\val a)),(\{1,2\},(\val b,\val b))\}
\end{equation}
and the weight vector reads $f=(3,4,6,2,-2,-4,1,1)\in\R^T$ (where the ordering is again given by~\eqref{eq:tuples_example}). Thus, the objective value of WCSP~$f$ for $x=(\val a,\val b)$ is $\VAL fx=3+2-4=1$. The upper bound equals~$\UB(f)=4+6+1=11$ and is tight because the CSP~$A^*(f)$ is satisfiable (recall Theorem~\ref{th:ub}). In particular, $\VAL{f}{\val b,\val a}=11$. 
\end{example}

\begin{figure}
\centering
\begin{subfigure}{0.3926\textwidth}
\centering
\includegraphics[width=\textwidth]{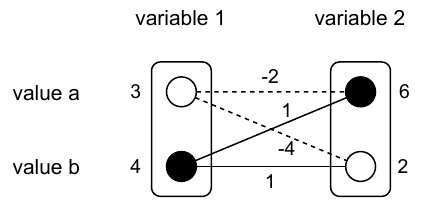}
\caption{WCSP~$f$.}
\label{fig:example_notation}
\end{subfigure}\hspace{1cm}
\begin{subfigure}{0.3926\textwidth}
\centering
\includegraphics[width=\textwidth]{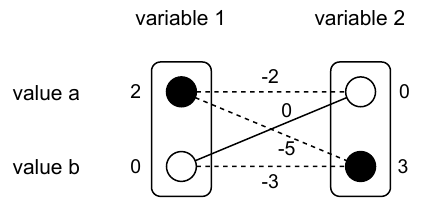}
\caption{WCSP~$d$.}
\label{fig:zeroprob_example}
\end{subfigure}
\caption{Visualisations of two WCSPs~$f$ and~$d$ with structure as in Example~\ref{ex:example_notation}. Variables (elements of~$\vars$) are depicted as rounded rectangles, tuples (elements of~$T$) as circles and line segments, and weights~$f_t$ (and~$d_t$) are written next to the circles and line segments.
Black circles and full lines indicate active tuples, whereas white nodes and dashed lines indicate non-active tuples.}
\label{fig:examples_intro}
\end{figure}

\subsection{Minimal Upper Bound over Reparametrizations}\label{sec:reparam}

We say that a WCSP $f\in\R^T$ is \emph{reparametrization\/} of a WCSP $g\in\R^T$ (also known as an equivalence-preserving transformation of~$g$)~\cite{trws,Wainwright08,schlesinger1976,Werner-PAMI07,Werner-PAMI-2010,Bogdan-book-2019,cooper2007optimal,Cooper-AI-2010,tourani2020taxonomy} if
\begin{equation}
\VAL fx=\VAL gx \quad \forall x\in\domain^\vars .
\label{eq:repar}
\end{equation}
That is, $f-g\in M^\perp$ where
\begin{equation}\label{eq:Mperp}
M^\perp = \ourset{d\in\R^T\mid \<d,\mu\>=0\; \forall \mu\in M}
= \ourset{d\in\R^T\mid \<d,\phi(x)\>=0\; \forall x\in \domain^\vars}
\end{equation}
is the orthogonal space \cite[Chapter~1]{zalinescu2002convex} of the set~\eqref{eq:M}. Here, `$d$' stands for `direction' but note that any $d\in M^\perp$, as a vector from $\R^T$, can be also seen as a standalone WCSP. The set~$M^\perp$ is a subspace of~$\R^T$, consisting of all WCSPs that have zero objective value for all assignments. Although $M^\perp$ is defined by an exponential number of equalities in~\eqref{eq:Mperp}, it has a simple, polynomial-sized description (for binary WCSP see \cite[\S{B}]{Werner-PAMI07}, for WCSPs of any arity see~\cite[\S3.2]{Werner-PAMI-2010}). An example of WCSP $d\in M^\perp$ is in Figure~\ref{fig:zeroprob_example}. The set of all reparametrizations of~$f$ is the affine subspace\footnote{Note, the symbol~`$+$' in the expression $f+M^\perp$ denotes the sum of a vector and a set of vectors.} $f+M^\perp=\ourset{f+d\mid d\in M^\perp}$. Clearly, the binary relation `is a reparametrization of' (on the set of WCSPs with a fixed structure) is reflexive, transitive and symmetric, hence an equivalence.

Given a WCSP $g\in\R^T$, it is a natural idea to minimize the upper bound on its optimal value by reparametrizations:
\begin{equation}
\min\ourset{ \UB(f) \mid \text{$f$ is a reparametrization of~$g$} }
\;\;=\;\;
\min_{f\in g+M^\perp} \UB(f) .
\label{eq:LPrelax}
\end{equation}
By introducing auxiliary variables (as in Footnote~\ref{foot:toLP}), this problem can be transformed to a linear program, which is the \emph{dual LP relaxation\/} of the WCSP~$g$~\cite{schlesinger1976,Werner-PAMI07,Werner-PAMI-2010,Bogdan-book-2019}. Every $f$~feasible for~\eqref{eq:LPrelax} satisfies
\begin{equation}
\UB(f) \ge \VAL fx = \VAL gx \quad \forall x\in D^V ,
\label{eq:LPrelax-ineq}
\end{equation}
i.e., $\UB(f)$ is an upper bound on the optimal value $\max_x\VAL gx$ of WCSP~$g$.
If inequality~\eqref{eq:LPrelax-ineq} holds with equality for some~$x$, then $f$~is optimal for~\eqref{eq:LPrelax} and the LP relaxation is tight. Necessary and sufficient conditions for optimality can be obtained from complementary slackness, see~\cite{Werner-PAMI07,Werner-PAMI-2010,Bogdan-book-2019}.

Problem~\eqref{eq:LPrelax} has been widely studied \cite{schlesinger1976,Werner-PAMI07,cooper2007optimal,Cooper-AI-2010,Zivny-VCSPbook-2012,Bogdan-book-2019,batra2011tighter} and many approaches for its (approximate) large-scale optimization have been proposed, typically based on block-coordinate descent \cite{mplp,tourani2018mplp++,tourani2020taxonomy,Werner-PAMI07,trws,komodakis2008beyond,sontag2009tree,sontag2008tightening} or constraint propagation~\cite{Cooper-AI-2010,Koval76,Werner-PAMI07,cooper2004cyclic,nguyen2017triangle}. 
If $f$~is optimal for~\eqref{eq:LPrelax}, then the CSP $\CSP^*(f)$ has a non-empty pairwise-consistency (PWC) closure (for binary WCSPs, PWC reduces to arc consistency) \cite{Werner-PAMI-2010}. We conjecture that PWC is in general the strongest level of local consistency of $\CSP^*(f)$ that can be achieved by reparametrizations without enlarging the WCSP structure (i.e., without introducing new weight functions).

We remark that some approaches \cite{cooper2007optimal,Cooper-AI-2010} achieve only (generalized) arc consistency rather than PWC because they optimize over a subset of all possible reparametrizations corresponding to a subspace of~$M^\perp$. In this case, WCSPs~$f$ optimal for~\eqref{eq:LPrelax} have been called \emph{optimally soft arc consistent (OSAC)}.

\subsection{Minimal Upper Bound over Super-Reparametrizations}\label{sec:super-reparam}

We say that a WCSP $f\in\R^T$ is a \emph{super-reparametrization\/}\footnote{Super-reparametrizations were called \emph{virtual potentials\/} in~\cite{komodakis2008beyond} and \emph{sup-reparametrizations\/} in~\cite{sontag2009tree}.} of a WCSP $g\in\R^T$ if
\begin{equation}
\VAL fx \ge \VAL gx \quad \forall x\in\domain^\vars .
\label{eq:suprepar}
\end{equation}
That is, $f-g\in M^*$ where
\begin{equation}\label{eq:M*}
M^* = \ourset{d\in\R^T\mid \<d,\mu\> \ge0\; \forall \mu\in M } 
= \ourset{d\in\R^T\mid \<d,\phi(x)\> \ge0\; \forall x\in \domain^\vars}
\end{equation}
is the dual cone \cite[Chapter 1]{zalinescu2002convex} to the set~\eqref{eq:M}. It is a polyhedral convex cone, consisting of the WCSPs that have nonnegative objective value for all assignments. This cone contains a line because $M^\perp\subseteq M^*$ and the subspace~$M^\perp$ is non-trivial (assuming $|V|>1$). Precisely, we have $M^*\cap(-M^*)=M^\perp$ where~$-M^*=\{-d\mid d\in M^*\}$. The set of all super-reparametrizations of~$f$ is the translated cone $f+M^*=\ourset{f+d\mid d\in M^*}$. For a given $d\in\R^T$, deciding whether $d\notin M^*$ is NP-complete, as shown later in Corollary~\ref{co:NPh_cone}.

The binary relation `is a super-reparametrization of' (on the set of  WCSPs with a fixed structure) induced by the convex cone $M^*$ is reflexive and transitive, hence a preorder. It is not antisymmetric: $f-g\in M^*$ and $g-f\in M^*$ does not imply $f=g$ but merely $f-g\in M^\perp$, i.e., that $f$~is a reparametrization of~$g$. This is because the cone~$M^*$ may contain a line, see~\cite[\S2]{jahn2011new} and~\cite[\S2.4]{boyd2004convex}.

\begin{remark}\label{re:marginal_cone}
The optimal value~\eqref{eq:optimization_task} of a WCSP~$f$ can be also written as
\begin{equation}\label{eq:optimization_task_M}
\max_{\mu\in M} \<f,\mu\>
= \max_{\mu\in\conv M} \<f,\mu\>
\end{equation}
where $\conv$ denotes the convex hull operator \cite{boyd2004convex}. The equality in~\eqref{eq:optimization_task_M} follows from the well-known fact that a linear function on a polytope attains its maximum in at least one vertex of the polytope \cite[Corollary~3.44]{Bogdan-book-2019}. The set $\conv M\subseteq[0,1]^T$ is known as the \emph{marginal polytope} and has the central role in approaches to WCSP based on linear programming (see \cite{Wainwright08,Bogdan-book-2019} and references therein). It is easy to show that 
\begin{equation}
M^* = (\conv M)^* = (\cone M)^*
\end{equation}
where $\cone$ denotes the conic hull operator \cite{boyd2004convex} and $^*$ the dual cone operator. Thus, \eqref{eq:M*}~can also be seen as the dual cone to the marginal polytope which, to the best of our knowledge, has not been mentioned before.
\end{remark}

Following~\cite{komodakis2008beyond}, we consider the problem
\begin{equation}
\min\ourset{ \UB(f) \mid \text{$f$ is a super-reparametrization of~$g$} }
\;\;=\;\;
\min_{f\in g+M^*} \UB(f).
\label{eq:LP}
\end{equation}
Again, this can be reformulated as a linear program. Every~$f$ feasible for~\eqref{eq:LP} (i.e., every super-reparametrization of~$g$) satisfies
\begin{equation}\label{eq:LP-ineq}
\UB(f) \ge \VAL fx \ge \VAL gx \qquad \forall x\in \domain^\vars.
\end{equation}
The next theorem characterizes optimal solutions:

\begin{theorem}\label{th:optimality_cond}
Let $f$ be feasible for~\eqref{eq:LP}. The following are equivalent:
\begin{enumerate}
\item[(a)] $f$ is optimal for~\eqref{eq:LP}.
\item[(b)] $\displaystyle \UB(f) = \max_{x\in \domain^\vars} \VAL fx = \max_{x\in \domain^\vars} \VAL gx$
\item[(c)] CSP~$A^*(f)$ has a solution~$x$ satisfying $\VAL fx=\VAL gx$.
\end{enumerate}
\end{theorem}

\begin{proof}
 (a)$\Leftrightarrow$(b): Denote $m=\max_x\VAL gx$, which by~\eqref{eq:LP-ineq} implies $B(f)\ge m$. To see that this bound is attained, define~$f$ by $f_t=m/|\constraints|$ for all $t\in T$. It can be checked from~\eqref{eq:energy} and~\eqref{eq:def_ub} that $B(f)=\VAL fx=m$ for all~$x$, so $f$~is feasible and optimal.

 (b)$\Rightarrow$(c): Since every feasible~$f$ satisfies~\eqref{eq:LP-ineq}, (b)~implies $\UB(f)=\VAL fx=\VAL gx$ for some~$x$. By Theorem~\ref{th:ub}(b), this implies~(c).

 (c)$\Rightarrow$(b): By Theorem~\ref{th:ub}(b) together with~\eqref{eq:LP-ineq}, (c)~implies $\UB(f)=\VAL fx=\VAL gx$ for some~$x$. Statement~(b) now follows from Theorem~\ref{th:ub}(a).
\end{proof}

Theorem~\ref{th:optimality_cond} in particular says that the optimal value of~\eqref{eq:LP} is equal to the optimal value of WCSP~$g$ (this has been observed already in~\cite[Theorem~1]{komodakis2008beyond}). As stated in~\cite{komodakis2008beyond}, this is not surprising because the complexity of the WCSP is hidden in the exponential set of constraints of~\eqref{eq:LP}. Let us remark that for $f\in g+M^*$, deciding whether $f$~is optimal for~\eqref{eq:LP} is NP-complete, as shown later in Corollary~\ref{co:NPh_optimal}.

Theorem~\ref{th:optimality_cond} has a simple corollary:

\begin{theorem}\label{th:CSPcert}
Let $g\in\R^T$. CSP $\CSP^*(g)$ is satisfiable if and only if $\UB(g)\le\UB(f)$ for every $f\in g+M^*$.
\end{theorem}
\begin{proof}
By Theorem~\ref{th:optimality_cond}, $\CSP^*(g)$ is satisfiable if and only if~\eqref{eq:LP} attains its optimum at the point $f=g$, i.e., $\UB(g)\le\UB(f)$ for every $f\in g+M^*$.
\end{proof}

\section{Iterative Method to Improve the Bound by\\Super-Reparametrizations}\label{se:optimizing_bound}

In this section, we present an iterative method to suboptimally solve~\eqref{eq:LP}. Starting from a feasible solution to~\eqref{eq:LP}, every iteration finds a new feasible solution with a lower objective, which by~\eqref{eq:LP-ineq} corresponds to decreasing the upper bound on the optimal value of the initial WCSP.

\subsection{Outline of the Method}

Consider a WCSP~$f$ feasible for~\eqref{eq:LP}, i.e., $f\in g+M^*$. By Theorem~\ref{th:optimality_cond}, a necessary (but not sufficient) condition for~$f$ to be optimal for~\eqref{eq:LP} is that CSP $\CSP^*(f)$ is satisfiable. By Theorem~\ref{th:CSPcert}, $\CSP^*(f)$ is satisfiable if and only if $\UB(f)\le \UB(f')$ for all $f'\in f+M^*$. In summary, we have the following implications and equivalences:
\begin{equation}
\begin{array}{ccc}
\text{$f$ is optimal for~\eqref{eq:LP}} &\implies& \text{CSP $A^*(f)$ is satisfiable} \\[.5ex]
\Big\Updownarrow && \Big\Updownarrow \\[1.5ex]
\UB(f)\le \UB(f') \;\;\forall f'\in g+M^* &\implies& \UB(f)\le \UB(f') \;\;\forall f'\in f+M^*
\end{array}
\label{eq:diagram}
\end{equation}
The left-hand equivalence is just the definition of the optimum of~\eqref{eq:LP}, the right-hand equivalence is Theorem~\ref{th:CSPcert}, and the top implication follows from Theorem~\ref{th:optimality_cond}. The bottom implication independently follows from transitivity of super-reparametrizations, which says that $f'\in f+M^*$ implies $f'\in g+M^*$ (assuming $f\in g+M^*$).

Suppose for the moment that we have an oracle that, for a given $f\in\R^T$, decides if $\CSP^*(f)$ is satisfiable and if it is not, finds some $f'\in f+M^*$ such that $\UB(f')<\UB(f)$ (which exists by Theorem~\ref{th:CSPcert}). By transitivity of super-reparametrizations, such~$f'$ is feasible for~\eqref{eq:LP}. This suggests an iterative scheme to improve feasible solutions to~\eqref{eq:LP}. \begin{samepage}We initialize $f^0:=g$ and then for $k=0,1,2,\dots$ repeat the following iteration:
\begin{center}
\framebox{
\begin{minipage}{.92\linewidth}
If CSP $A^*(f^k)$ is satisfiable, stop.\\ Otherwise, find $f^{k+1}\in f^k+M^*$ such that $\UB(f^{k+1})<\UB(f^k)$.
\end{minipage}
}
\end{center}
\end{samepage}
Note that transitivity of super-reparametrizations implies $f^k\in f^0+M^*$ for every~$k$, so every~$f^k$ is feasible for~\eqref{eq:LP} as expected.
An example of a single iteration is shown in Figure~\ref{fig:example_iteration:a} and~\ref{fig:example_iteration:b}.

\begin{figure}[t]
\begin{subfigure}{0.49\textwidth}
\centering
\includegraphics[width=0.848979\textwidth]{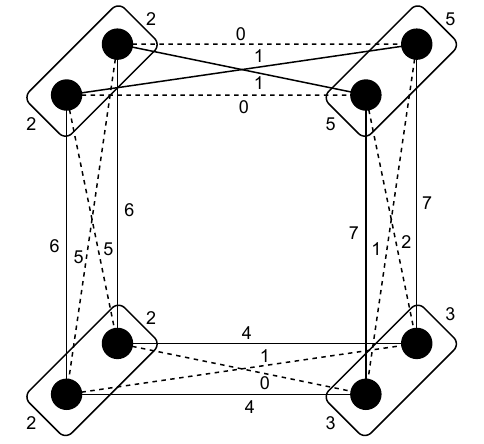}
\caption{WCSP~$f^0$, $A^*(f^0)$ unsatisfiable.}\label{fig:example_iteration:a}
\end{subfigure}
\begin{subfigure}{0.49\textwidth}
\centering
\includegraphics[width=0.848979\textwidth]{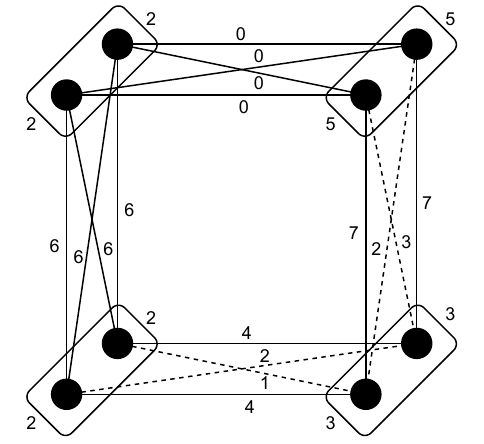}
\caption{WCSP~$f^1$, $\UB(f^1)<\UB(f^0)$.}\label{fig:example_iteration:b}
\end{subfigure}

\begin{subfigure}{\textwidth}
\centering
\includegraphics[width=0.416\textwidth]{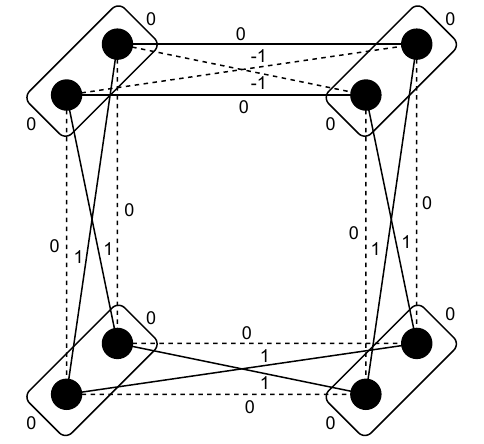}
\caption{Certificate~$d$ of unsatisfiability of~$A^*(f^0)$, $f^1=f^0+d$.}\label{fig:example_iteration:c}
\end{subfigure}

\caption{Example of one iteration on a binary WCSP whose (hyper)graph is a cycle of length 4.}
\label{fig:example_iteration}
\end{figure}

This iterative method belongs to the class of local search methods to solve~\eqref{eq:LP}: having a current feasible estimate~$f^k$, we search for the next estimate~$f^{k+1}$ with a strictly better objective within a neighborhood $f^k+M^*$ of~$f^k$. 
We can define \emph{local optima} of~\eqref{eq:LP} with respect to this method to be super-reparametrizations~$f$ of~$g$ such that $\CSP^*(f)$ is satisfiable.

\subsubsection{Properties of the Method}\label{se:properties_of_the_method}

\begin{figure}
    \centering
    \begin{picture}(310,220)
    \put(10,0){\includegraphics[width=300pt]{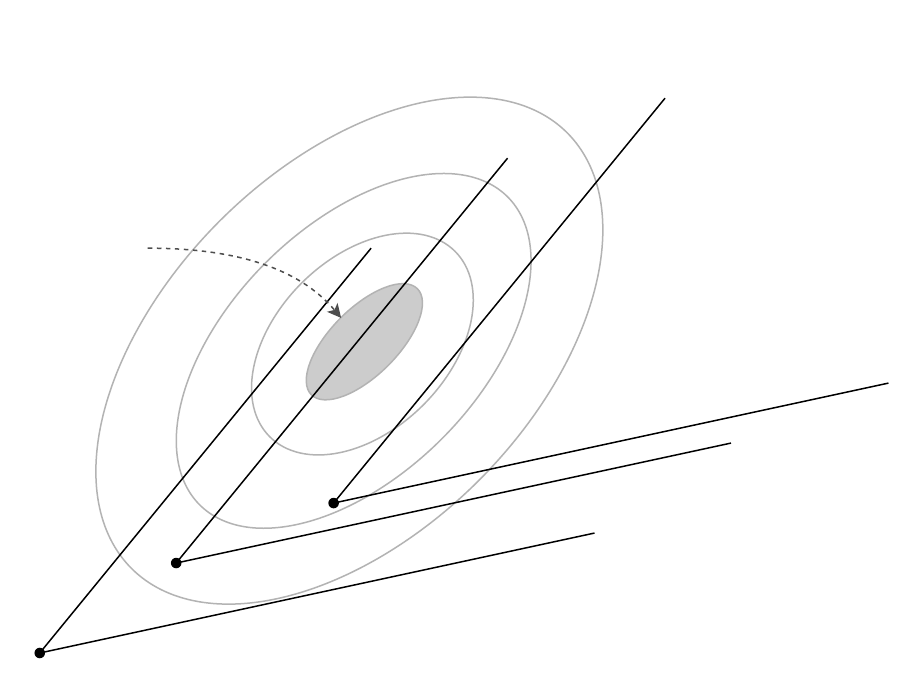}}
    \put(110,65){\small $f^2$}
    \put(57,43){\small $f^1$}
    \put(0,22){\small $f^0=g$}
    \put(0,150){\small $\argmin\limits_{f\in g+M^*} \UB(f)$}
    \end{picture}
    \caption{The shrinking of the search space of the iterative method. The figure illustrates the translated cones~$f^i+M^*$ and several contours of the objective~$\UB(f)$. After the second iteration, all global minima of the original problem (marked in grey) become inaccessible as the right hand side of~\eqref{eq:shrinking-gap} increases.}
    \label{fig:nested_cones_contours}
    \end{figure}
\begin{figure}
    \centering
    \begin{picture}(300,206.4)
    \put(0,11){\includegraphics[width=264pt]{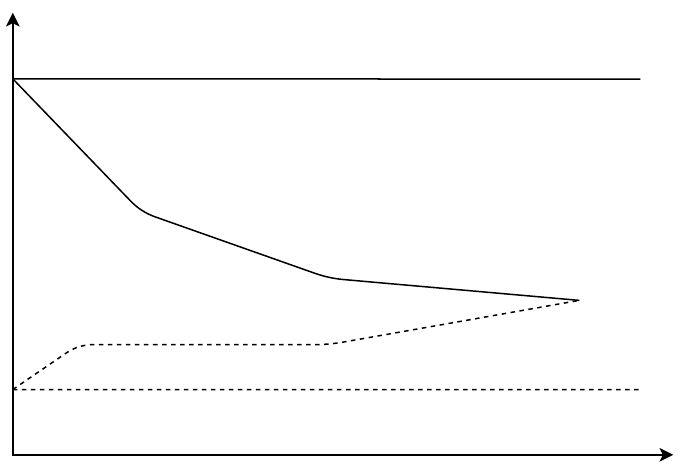}} 
    \put(66,108){\small $\UB(f^k)$}
    \put(11,69){\small $\displaystyle \min_{f\in f^k+M^*}\UB(f)=\max_x\VAL{f^k}x$}
    \put(58.8,32.4){\small $\displaystyle \min_{f\in g+M^*}\UB(f)=\max_x\VAL gx$}
    \put(117.6,165.6){\small $\UB(g)$}
    \put(260.4,14.4){\small iteration~$k$}
    \end{picture}
    \caption{Illustration to the iterative scheme: $\UB(g)$~and~$\UB(f^k)$ are shown by the full lines, $\max_x\VAL gx$~and~$\max_x\VAL{f^k}x$ are represented by the dashed lines.}
    \label{fig:plot_values}
\end{figure}

By transitivity of super-reparametrizations, for every~$k$ we have
\begin{equation}
f^{k+1}+M^* \subseteq f^k+M^*
\label{eq:cone<cone}
\end{equation}
which holds with equality if and only if $f^{k+1}\in f^k+M^\perp$ (i.e.,  $f^{k+1}$~is a reparametrization of~$f^k$). This shows that the search space of the method may shrink with increasing~$k$, in other words, a larger and larger part of the feasible set $f^0+M^*$ of~\eqref{eq:LP} is cut off and becomes forever inaccessible. If, for some~$k$, all (global) optima of~\eqref{eq:LP} happen to lie in the cut-off part, the method has lost any chance to find a global optimum. This is illustrated in Figure~\ref{fig:nested_cones_contours}.

This has the following consequence. Every~$f^k$ satisfies
\begin{equation}
\UB(f^k) \;\;\ge\;\; \min_{f\in f^k+M^*} \UB(f) = \max_{x\in D^V} \VAL{f^k}x .
\label{eq:shrinking-gap}
\end{equation}
In every iteration, the left-hand side of inequality~\eqref{eq:shrinking-gap} decreases and the right-hand side increases or stays the same due to~\eqref{eq:cone<cone}. If both sides meet for some~$k$, the CSP $\CSP^*(f^k)$ becomes satisfiable by Theorem~\ref{th:ub}(b) and the method stops. Monotonic increase of the right-hand side can be seen as `greediness' of the method: if we could choose $f^{k+1}$ from the initial feasible set $f^0+M^*$ rather than from its subset $f^k+M^*$, the right-hand side could also decrease. Any increase of the right-hand side is undesirable because the bounds~$\UB(f^k)$ in future iterations will never be able to get below it. This is illustrated in Figures~\ref{fig:plot_values} and~\ref{fig:finalexample}. Unlike in~\eqref{eq:LPrelax}, note that not every optimal assignment for WCSP~$f$ is optimal for WCSP~$g$. We will return to this in~\S\ref{se:theoretical_properties}.

If $A^*(f^k)$ is unsatisfiable, there are usually many vectors $f^{k+1}\in f^k+M^*$ satisfying $\UB(f^{k+1})<\UB(f^k)$. We should choose among them the one that does not cause `too much' shrinking of the search space and/or increase of the right-hand side of~\eqref{eq:shrinking-gap}. Inclusion~\eqref{eq:cone<cone} holds with equality if and only if $f^{k+1}\in f^k+M^\perp$, so whenever possible we should choose $f^{k+1}$ to be a reparametrization (rather than just a super-reparametrization) of~$f^k$. Unfortunately, we know of no other useful theoretical results to help us choose~$f^{k+1}$, so we are left with heuristics. One natural heuristic is to choose~$f^{k+1}$ such that the vector $f^{k+1}-f^k$ is sparse (i.e., has only a small number of non-zero components) and its positive components are small. Unfortunately, this can sometimes be too restrictive because, e.g., vectors from~$M^\perp$ can be dense and their components have unbounded magnitudes.

\begin{figure}
\centering
\begin{subfigure}[t]{0.3926\textwidth}
\centering
\includegraphics[width=\textwidth]{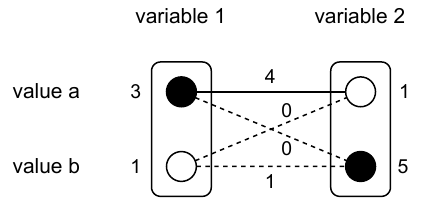}
\caption{WCSP~$g$.}
\end{subfigure}
\begin{subfigure}[t]{0.3926\textwidth}
\centering
\includegraphics[width=\textwidth]{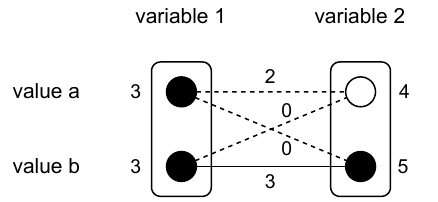}
\caption{WCSP~$f$.}
\end{subfigure}
\caption{WCSP~$f$ is a super-reparametrization of WCSP~$g$ and this pair of WCSPs satisfies $B(f)=11<B(g)=12$ and $\max_{x\in\domain^\vars} \VAL fx=11>\max_{x\in\domain^\vars} \VAL gx=8$. Assignment $x=(\val{b,b})$ is not optimal for~$g$ despite that $B(f)=\VAL fx$. The fact that~$f$ is a super-reparametrization of~$g$ can be verified by computing the objective value for each assignment, e.g., for assignment $x=({\val a},{\val a})$ we have $\VAL gx=3+4+1=8 \le \VAL fx=3+2+4=9$.}
\label{fig:finalexample}
\end{figure}

\subsubsection{Employing Constraint Propagation}\label{se:constraint_propagation_scheme}

So far we have assumed we can always decide if CSP $A^*(f)$~is satisfiable. This is unrealistic because the CSP is NP-complete.
Yet the approach remains applicable even if we detect unsatisfiability of~$A^*(f)$ only sometimes, e.g., using constraint propagation. Then our iteration changes to:
\begin{center}
\framebox{
\begin{minipage}{.92\linewidth}
Try to prove that CSP $A^*(f^k)$ is unsatisfiable.\\
If we succeed, find $f^{k+1}\in f^k+M^*$ such that $\UB(f^{k+1})<\UB(f^k)$.\\
If we fail, stop.
\end{minipage}
}
\end{center}
In this case, stopping points of the method will be even weaker local minima of~\eqref{eq:LP}, but they nevertheless might be still non-trivial and useful.

In the sequel we develop this approach in detail. In particular we show, if $\CSP^*(f^k)$ is unsatisfiable, how to find a vector $f^{k+1}\in f^k+M^*$ satisfying $\UB(f^{k+1})<\UB(f^k)$. We will do it in two steps. First (in~\S\ref{se:improving_directions}), given the CSP $\CSP^*(f^k)$ we find a direction $d\in M^*$ using constraint propagation. This direction is a \emph{certificate of unsatisfiability} of the CSP $\CSP^*(f^k)$ and, at the same time, an \emph{improving direction} for~\eqref{eq:LP}. Second (in~\S\ref{se:line_search}), given~$d$ and $f^k$, we find a \emph{step size} $\alpha>0$ such that $f^{k+1}=f^k+\alpha d$ and $B(f^k)>B(f^{k+1})$. An example of such a certificate of unsatisfiability is shown in Figure~\ref{fig:example_iteration:c}.

\subsubsection{Relation to Existing Approaches}

The Augmenting DAG algorithm~\cite{Koval76,Werner-PAMI07} and the VAC algorithm~\cite{Cooper-AI-2010} are (up to the precise way of computing certificates~$d$ and step sizes~$\alpha$) an example of the described approach, which uses arc consistency to attempt to prove unsatisfiability of~$A^*(f^k)$. In this favorable case, there exist certificates~$d\in M^\perp$, so we are, in fact, applying local search to~\eqref{eq:LPrelax} rather than~\eqref{eq:LP}. For stronger local consistencies, such certificates, in general, do not exist (i.e., inevitably $\VAL dx>0$ for some~$x$).

The algorithm proposed in~\cite{komodakis2008beyond} can be also seen as an example of our approach. It interleaves iterations using arc consistency (in fact, the Augmenting DAG algorithm) and iterations using cycle consistency.

As an alternative to our approach, stronger local consistencies can be achieved by introducing new weight functions (of possibly higher arity) into the WCSP objective~\eqref{eq:energy} and minimizing an upper bound by reparametrizations, as in~\cite{sontag2008tightening,batra2011tighter,Werner-PAMI-2010,Werner-PAMI-2015,nguyen2017triangle}. In our particular case, after each update $f^{k+1}=f^k+\alpha d$ we could introduce a new weight function with scope
\begin{equation}
\scope' = \bigcup\ourset{\scope \mid (S,k)\in T, \; d_S(k)\neq 0}
\end{equation}
and weights
\begin{equation}
f_{\scope'}(k) = -\alpha\sum_{\substack{\scope \in \constraints \\ \scope \subseteq \scope'}} d_\scope(k[\scope])
\end{equation}
where $k\in\domain^{\scope'}$. Notice that such an added weight function would not increase the bound~\eqref{eq:def_ub} since its weights are non-positive due to the fact that it needs to decrease the objective value for some assignments. In this view, our approach can be seen as enforcing stronger local consistencies but omitting these compensatory higher-order weight functions, thus saving memory.

Finally, the described approach can be seen as an example of the primal-dual approach~\cite{werner-dlask-maxsat} to optimize linear programs using constraint propagation. In detail, \cite{werner-dlask-maxsat}~proposed to construct the complementary slackness system for a given feasible solution and apply constraint propagation to detect if the system is satisfiable. If it is not satisfiable, this implies the existence of a certificate of unsatisfiability that can be used to improve the current solution. In our particular case, if~\eqref{eq:LP} is formulated as a linear program, then the complementary slackness conditions (expressed in terms of the dual variables) are equivalent to the optimality conditions stated in Theorem~\ref{th:optimality_cond} expressed as a set of linear equalities with an exponential number of non-negative variables. Applying constraint propagation on this system is in correspondence with constraint propagation on a CSP.

\subsection{Certificates of Unsatisfiability of CSP}\label{se:improving_directions}

Constraint propagation\footnote{We speak only about constraint propagation but the approach outlined in this section is applicable to any method that proves unsatisfiability of a CSP by iteratively forbidding subsets of tuples. In theory, as a stronger alternative one could also use any CSP solver that is augmented to provide a certificate of unsatisfiability (which is always possible, as we will discuss later in this section).\label{foo:generalizations_of_CP}} is an iterative algorithm, which in each iteration (executed by a \emph{propagator}) infers that some allowed tuples $R\subseteq A$ of a current CSP $A\subseteq T$ can be forbidden without changing its solution set, i.e., $\SOL(A)=\SOL(A-R)$, and forbids these tuples, i.e., sets $A:=A-R$. The algorithm terminates when it is no longer able to forbid any tuples (in which case the propagator returns $R=\emptyset$) or when it becomes explicit that the current CSP is unsatisfiable. The former usually happens when the CSP achieves some \emph{local consistency} level~$\Phi$. The latter happens if $A\cap T_S=\emptyset$ for some $S\in C$, which implies unsatisfiability of~$\CSP$ because\footnote{If $|\scope|=1$, this event is often called a `domain wipe-out'.} every assignment has to use  one tuple from each~$T_S$.

In this section, we show how to augment constraint propagation so that if it proves a CSP unsatisfiable, it also provides its \emph{certificate of unsatisfiability} $d\in M^*$. This certificate is needed as an improving direction for~\eqref{eq:LP}, as was mentioned in~\S\ref{se:constraint_propagation_scheme}.
First, in~\S\ref{se:deactivating_vectors}, we introduce a more general concept, \emph{deactivating directions}. One iteration of constraint propagation constructs an $R$-deactivating direction for the current CSP~$A$, which certifies that $\SOL(A)=\SOL(A-R)$. Then, in~\S\ref{se:composing}, we show how to \emph{compose} the deactivating directions obtained from individual iterations of constraint propagation to a single deactivating direction for the initial CSP. If the initial CSP has been proved unsatisfiable by the propagation, this composed deactivating direction is then its certificate of unsatisfiability.

\subsubsection{Deactivating Directions}\label{se:deactivating_vectors}

\begin{definition}\label{de:t-deactivating}
Let $\CSP\subseteq T$ and $\deactivating \subseteq \CSP$, $\deactivating\neq\emptyset$. An $\deactivating$-\emph{deactivating direction for CSP\/}~$\CSP$ is a vector $d \in M^*$ satisfying
\begin{enumerate}
    \item[(a)] $d_t < 0$ for all $t \in \deactivating$,
    \item[(b)] $d_t = 0$ for all $t \in \CSP-\deactivating$.
\end{enumerate}
\end{definition}

For fixed $\CSP$ and~$\deactivating$, all $\deactivating$-deactivating directions for~$\CSP$ form a convex cone. Here, we show one way of constructing a deactivating direction:

\begin{theorem}\label{th:deact_exists}
Let $R\subseteq \CSP\subseteq T$ be such that $\SOL(\CSP)=\SOL(\CSP-R)$ 
and~$\deactivating\neq\emptyset$.
Denote\,\footnote{The quantity $\delta>0$ is the number of scopes~$\scope$ such that $\TPA{\scope}$~contains at least one tuple from~$\deactivating$. In other words, for every assignment $x\in\domain^\vars$, $(\scope,x[\scope])\in \deactivating$ holds for at most~$\delta$ scopes. We remark that the value of~$\delta$ could be in some cases decreased, thus decreasing also the objective values~$\VAL dx$. However, deciding whether~\eqref{eq:deact_exists} is not an $\deactivating$-deactivating direction for~$\CSP$ for a given value~$\delta$ and~$\CSP$ is an NP-complete problem, see Theorem~\ref{th:NPh_deac}.}
\begin{equation}
 \delta=\ourabs{\ourset{\scope \in \constraints \mid \TPA{\scope}\cap \deactivating\neq \emptyset}}.
\end{equation}
Then vector $d\in\R^T$ with components
\begin{equation}\label{eq:deact_exists}
    d_t = \begin{cases}
    -1 & \text{if $t \in \deactivating$}\\
    \delta & \text{if $t \in T-A$}\\
    0 & \text{otherwise (i.e., $t\in A-R$)}
    \end{cases}
\end{equation}
is an $\deactivating$-deactivating direction for~$\CSP$.
\end{theorem}

\begin{proof}
Conditions (a) and (b) of Definition~\ref{de:t-deactivating} are clearly satisfied, so it only remains to show that $d\in M^*$. We have
\begin{equation}\label{eq:deac_energy}
\VAL dx
= \sum_{t\in T} d_t\phi_t(x)
= \sum_{t\in \deactivating}-\phi_t(x) + \sum_{t\in T-A}\delta\phi_t(x)
= -n_1(x)+\delta n_2(x)
\end{equation}
where $n_1(x)=\ourabs{\ourset{\scope\in\constraints \mid (\scope,x[\scope])\in \deactivating}}$ and $n_2(x)=\ourabs{\ourset{\scope\in\constraints \mid (\scope,x[\scope])\in T-A}}$.

For contradiction, let $x\in D^V$ satisfy $\VAL dx<0$. This implies $n_1(x)>0$ and $n_2(x)=0$, where the latter is because $n_1(x)\le\delta$ by the definition of~$\delta$. That is, we have $(\scope^*,x[\scope^*])\in \deactivating$ for some $\scope^*\in\constraints$ and $(\scope,x[\scope])\in A$ for all~$\scope\in\constraints$. But the latter means $x\in \SOL(A)$ and the former implies $x\notin \SOL(A-\deactivating)$, a contradiction.
\end{proof}

\begin{theorem}\label{th:deact_implies}
Let $\CSP\subseteq T$ and $\deactivating \subseteq \CSP$. If there exists an $\deactivating$\nobreakdash-deactivating direction for~$\CSP$, then $\SOL(\CSP)=\SOL(\CSP-\deactivating)$.
\end{theorem}
\begin{proof} Observe that $\SOL(\CSP)=\SOL(\CSP-\deactivating)$ is equivalent to $\SOL(\CSP)\subseteq\SOL(\CSP-\deactivating)$ because forbidding tuples may only remove solutions, i.e., $\SOL$ is an isotone map (see~\S\ref{se:minimal_csp}).

Let~$d$ be an $\deactivating$-deactivating direction for~$\CSP$ and let~$x\in\SOL(\CSP)-\SOL(\CSP-\deactivating)$, so $(\scope,x[\scope])\in\deactivating$ for some~$\scope\in\constraints$. By~\eqref{eq:energy'}, we have $\VAL d{x}< 0$ because $d_\scope(x[\scope]) = 0$ for all $(\scope,x[\scope])\in \CSP-\deactivating$ by condition~(b) in~Definition~\ref{de:t-deactivating} and $d_{\scope}(x[\scope]) < 0$ for all $(\scope,x[\scope])\in \deactivating$ by condition~(a). This contradicts~$d\in M^*$.
\end{proof}

Combining Theorems~\ref{th:deact_exists} and~\ref{th:deact_implies} yields that for any~$R\subseteq A$ with~$R\neq\emptyset$, an $\deactivating$\nobreakdash-deactivating direction for~$\CSP$ exists if and only if $\SOL(\CSP)=\SOL(\CSP-\deactivating)$. Thus, any $\deactivating$-deactivating direction for~$A$ is a \emph{certificate} of the fact that $\SOL(\CSP)=\SOL(\CSP-\deactivating)$.

Unfortunately, vectors~$d$ calculated naively by~\eqref{eq:deact_exists} can have many non-zero components, which is undesirable as explained in~\S\ref{se:properties_of_the_method}. 
However, it is clear from Definition~\ref{de:t-deactivating} that if $A\subseteq A'\subseteq T$ and $d$~is an $R$-deactivating direction for~$A'$, then $d$~is an $R$-deactivating direction also for~$A$. Moreover, \eqref{eq:deact_exists}~shows that larger sets~$\CSP$ give rise to sparser vectors~$d$. This offers us a possibility to obtain a sparser $\deactivating$\nobreakdash-deactivating direction for~$\CSP$ if we can provide a superset $\CSP'\supseteq\CSP$ of the allowed tuples satisfying $\SOL(\CSP')=\SOL(\CSP'-R)$.

Given $A\subseteq T$ and $R\subseteq A$, finding a maximal (w.r.t.\ the partial ordering by inclusion) superset $A'\supseteq A$ such that $\SOL(A')=\SOL(A'-R)$ is closely related to finding a minimal unsatisfiable core\footnote{ Given an unsatisfiable CSP $A\subseteq T$, finding a maximal set $A'\supseteq A$ such that $A'$~is still unsatisfiable corresponds to finding a minimally unsatisfiable set of tuples~\cite{gregoire2007must}. This is a finer-grained (tuple-based rather than constraint-based) version of finding a minimal unsatisfiable core of a CSP~\cite{gregoire2008finding}. Note that we are looking here for a \emph{maximal} superset $A'$ in contrast to a \emph{minimal} unsatisfiable core because we define CSP instances by \emph{allowed} tuples while  cores are CSP instances defined by \emph{forbidden} tuples.} of an unsatisfiable CSP. While finding a maximal such subset is very likely intractable\footnote{The problem of finding a minimal unsatisfiable core has been designated in~\cite{gregoire2008finding} to be `highly intractable' based on results from~\cite{papadimitriou1985complexity}.}, for obtaining a `sparse enough' vector~$d$ it suffices to find a `large enough' such superset~$A'$.
Such a superset is often cheaply available as a side result of executing the propagator. Namely, we take $A'=T-P$ where $P$~is the set of forbidden tuples that were visited during the run of the propagator. Clearly, tuples not visited by the propagator could not be needed to infer $\SOL(\CSP)=\SOL(\CSP-R)$.
Note that $P$~need not be the same for each CSP instance, even for a fixed level of local consistency: for example, if the arc consistency closure of~$\CSP$ is empty, then $\CSP$~is unsatisfiable but a domain wipe-out may occur sooner or later depending on~$\CSP$, which affects which tuples needed to be visited.

Let us emphasize that an $\deactivating$-deactivating direction for $\CSP$ need not be always obtained using formula~\eqref{eq:deact_exists}, any other method can be used as long as $d$~satisfies Definition~\ref{de:t-deactivating}.
We will now give examples of deactivating directions corresponding to some popular constraint propagation rules. In these examples, we assume that our CSP contains all unary constraints (i.e., $\{i\}\in\constraints$ for each $i\in\vars$), so that rather than deleting domain values we can forbid tuples of the unary constraints.

\begin{example}\label{ex:ac}
Let us consider (generalized) \emph{arc consistency\/} (AC). A CSP~$\CSP$ is (G)AC if for all $\scope\in\constraints$, $i\in\scope$ and $k\in D$ we have the equivalence\,\footnote{Note, for convenience we use a slightly unusual definition of arc consistency, allowing to restrict not only domains but also constraint relations. This definition was also considered in~\cite[\S6]{Bogdan-book-2019} or~\cite{Werner-PAMI-2010}.}
\begin{equation}
(\{i\},k)\in \CSP \quad\iff\quad (\exists l\in D^S\colon (\scope,l)\in\CSP, \; l_i=k) .
\label{eq:GAC}
\end{equation}

If, for some $\scope\in\constraints$, $i\in\scope$ and $k\in D$, the left-hand statement in~\eqref{eq:GAC} is true and the right-hand statement is false, the AC propagator infers $\SOL(\CSP)=\SOL(\CSP-\deactivating)$ where $\deactivating=\{(\{i\},k)\}$. To infer this, it suffices to know that the tuples $P=\ourset{(\scope,l) \mid l\in D^S, \; {l_i=k}}$ are all forbidden. An $\deactivating$\nobreakdash-deactivating direction~$d$ for~$\CSP$ can be chosen as in~\eqref{eq:deact_exists} where $\delta=\ourabs{\ourset{\scope'\in\constraints \mid \TPA{\scope'}\cap \deactivating\neq \emptyset}}=1$ and $\CSP$~is replaced by~$T-P$. Note that then we have $d\in M^\perp$.

If the left-hand statement in~\eqref{eq:GAC} is false and the right-hand statement is true, the AC propagator infers $\SOL(\CSP)=\SOL(\CSP-\deactivating)$ where $\deactivating=\ourset{(\scope,l) \mid l\in D^S, \; l_i=k}\cap\CSP$. To infer this, it suffices to know that the tuple $P=\{(\{i\},k)\}$ is forbidden. In this particular case, rather than using~\eqref{eq:deact_exists} (with~$A$ replaced by~$T-P$), it is better to choose~$d$ as
\begin{equation}\label{eq:AC2}
    d_t = \begin{cases}
   -1        & \text{if $t \in \ourset{(\scope,l) \mid l\in D^S, \; l_i=k}$}  \\
   1        & \text{if $t\in P$}\\
   0        & \text{otherwise}
  \end{cases}.
\end{equation}
Vector~\eqref{eq:AC2} satisfies $d\in M^\perp$, in contrast to vector~\eqref{eq:deact_exists} which satisfies only $d\in M^*$. Thus, the update $f^{k+1}=f^k+\alpha d$ is a mere reparametrization, which is desirable as explained in~\S\ref{se:properties_of_the_method}.

We note that reparametrizations considered in the previous paragraphs correspond to soft arc consistency operations \emph{extend\/} and \emph{project\/} in~\cite{Cooper-AI-2010}.
\end{example}

\begin{example}\label{ex:cc}
We now consider \emph{cycle consistency\/} as defined in~\cite{komodakis2008beyond}.\footnote{This is different from \emph{cyclic consistency\/} as defined in~\cite{cooper2004cyclic}. E.g., reparametrizations are sufficient to enforce cyclic consistency, whereas super-reparametrizations are needed for cycle consistency. Cycle consistency is not common in the constraint programming community. It can be shown that if the graph~$(V,E)$ is complete, then path inverse consistency~\cite{freuder1996neighborhood,Bessiere-CP-handbook} corresponds to cycle consistency w.r.t. all cycles of length~3.} 
As this local consistency was defined only for binary CSPs, we assume that $|\scope|\le2$ for each $\scope\in\constraints$ and denote $\edges=\ourset{\scope\in\constraints \mid |\scope|=2}$, so that $(\vars,\edges)$ is an undirected graph. Let~$\mathcal{L}$ be a (polynomially sized) set of cycles in the graph~$(\vars,\edges)$. A CSP~$\CSP$ is cycle consistent w.r.t.~$\mathcal{L}$ if for each tuple $(\{i\},k)\in \CSP$ (where $i\in V$ and $k\in D$) and each cycle $L\in\mathcal{L}$ that passes through node $i\in \vars$, there exists an assignment~$x$ with~$x_i=k$ that uses only allowed tuples in cycle~$L$. It can be shown that the cycle repair procedure in~\cite{komodakis2008beyond} constructs a deactivating direction whenever an inconsistent cycle is found. Moreover, the constructed direction in this case coincides with~\eqref{eq:deact_exists} where~$A$ is replaced by~$T-P$ for a suitable set~$P$ that contains a subset of the forbidden tuples within the cycle.
\end{example}

\begin{example}\label{ex:sac}
Recall that a CSP~$\CSP$ is \emph{singleton arc consistent (SAC)\/} if for every tuple $t=(\{i\},k)\in A$ (where $i\in V$ and $k\in D$), the CSP\footnote{This can be also stated as $\CSP\vert_{x_i=k}= \CSP-\ourset{(\{i\},k')\mid k'\in \domain-\{k\}}$. In other words, the solutions of the CSP $\CSP\vert_{x_i=k}$ are the solutions~$x$ to CSP~$\CSP$ satisfying~$x_i=k$. This notation is used, e.g., in~\cite{bessiere2011efficient}.} $\CSP\vert_{x_i=k} =\CSP-(\TPA{\{i\}}-\{(\{i\},k)\})$  has a non-empty arc-consistency closure.
Good (i.e., sparse) deactivating directions for SAC can be obtained as follows. For some $(\{i\},k)\in \CSP$, we enforce arc consistency of CSP~$\CSP|_{x_i=k}$, during which we store the causes for forbidding each tuple. If $\CSP|_{x_i=k}$ is found to have empty AC closure, we backtrack and identify only those tuples which were necessary to prove the empty AC closure.
These tuples form the set~$P$. The deactivating direction is then constructed as in~\eqref{eq:deact_exists} where~$\deactivating=\{(\{i\},k)\}$ and $\CSP$~is replaced by~$T-P$. Note that SAC does not have bounded support as many other local consistencies~\cite{bessiere2008theoretical} do, so the size of~$P$ can be significantly different for different CSP instances. We show a detailed example of constructing a deactivating direction using SAC in Appendix~\ref{ap:example}.
\end{example}

\subsubsection{Composing Deactivating Directions}\label{se:composing}

Consider now a propagator which, for a current CSP $\CSP\subseteq T$, returns a set $\deactivating\subseteq\CSP$ such that $\SOL(\CSP)=\SOL(\CSP-\deactivating)$ and an $\deactivating$\nobreakdash-deactivating direction for~$\CSP$. This propagator is applied iteratively, each time forbidding a different set of tuples, until the current CSP achieves the desired local consistency level~$\Phi$ or it becomes explicit that the CSP is unsatisfiable. This is outlined in Algorithm~\ref{al:propagation_with_storing_vectors}, which stores the generated sets $\deactivating_i$ of tuples being forbidden and the corresponding $\deactivating_i$\nobreakdash-deactivating directions~$d^i$. By line~5 of the algorithm, we have $\CSP_i = \CSP-\bigcup_{j=0}^{i-1} \deactivating_j$ for every $i \in \{0, \ldots, n+1\}$. Therefore, by Theorem~\ref{th:deact_implies}, we have $\SOL(A)=\SOL(A_1)=\SOL(A_2)=\cdots=\SOL(A_{n+1})$, which implies that if $A_{n+1}$~is unsatisfiable then so is~$A$.

\begin{algorithm}[t]
\caption{The procedure {\tt propagate} applies constraint propagation to CSP $\CSP\subseteq T$ and returns the sequence $(\deactivating_i)_{i=0}^n$ of tuple sets that were forbidden and the corresponding deactivating directions $(d^i)_{i=0}^n$. If all tuples in some scope~$\scope \in \constraints$ become forbidden during propagation, {\tt propagate} returns also $\scope$,  otherwise it returns $S=\emptyset$.}\label{al:propagation_with_storing_vectors}
\begin{algorithmic}[1]
\State \textbf{procedure} $(\scope,(\deactivating_i)_{i=0}^n,(d^i)_{i=0}^n)=\text{\tt propagate}(A)$ 
\State Initialize $n:=0$, $\CSP_0:=\CSP$.
\While{\normalfont $\CSP_n$ is not $\Phi$-consistent}
   \State Find a set~$\deactivating_n\subseteq \CSP_n$ and an $\deactivating_n$-deactivating direction $d^n$ for $\CSP_n$.
   \State $\CSP_{n+1}:=\CSP_n-\deactivating_n$
   \If{$\exists \scope\in\constraints\colon \CSP_{n+1}\cap\TPA{\scope} = \emptyset$}
    \State \Return{$(\scope,(\deactivating_i)_{i=0}^n$, $(d^i)_{i=0}^n)$}
   \EndIf
   \State $n:=n+1$
\EndWhile
\State \Return{$(\emptyset,(\deactivating_i)_{i=0}^{n-1}$, $(d^i)_{i=0}^{n-1}))$}
\end{algorithmic}
\end{algorithm}

In this section, we show how to compose the generated sequence of $\deactivating_i$-deactivating directions~$d^i$ for~$\CSP_i$ into a single $\big(\bigcup_{i=0}^n \deactivating_i\big)$\nobreakdash-deactivating direction for~$\CSP$. This can be done using the following composition rule:

\begin{theorem}\label{th:combining_deactivating_vectors}
Let $\CSP\subseteq T$ and $\deactivating,\deactivating'\subseteq \CSP$ where $\deactivating\cap \deactivating'=\emptyset$. Let $d$~be an $\deactivating$-deactivating direction for~$\CSP$. Let $d'$~be an $\deactivating'$-deactivating direction for~$\CSP-\deactivating$. Let
\begin{equation}\label{eq:combining_delta}
    \delta = \begin{cases}
    0 & \text{if $d_t' \leq -1$ for all $t \in \deactivating$} ,\\
    \max\ourset{(-1-d_t')/d_t \mid t \in \deactivating, \; d_t' > -1} & \text{otherwise} .
    \end{cases}
\end{equation}
Then $d'' = d'+\delta d$ is an $(\deactivating\cup \deactivating')$-deactivating direction for~$\CSP$.
\end{theorem}
\begin{proof}
First, if $d_t' \leq -1$ for all $t \in \deactivating$, then $d''=d'$ satisfies the required condition immediately. Otherwise, $\delta > 0$ since $d_t < 0$ for all $t \in \deactivating$ by definition and $-1-d'_t < 0$ due to $d'_t>-1$ in the definition of~$\delta$. We will show that~$d''$ satisfies the conditions in Definition~\ref{de:t-deactivating}.

For $t \in \deactivating$ with $d'_t \leq -1$, $d''_t = d'_t + \delta d_t < d'_t \leq -1$ because $\delta d_t < 0$. If $t\in \deactivating$ and $d'_t > -1$, then $\delta \geq (-1-d'_t)/d_t$, so $d''_t = d_t'+\delta d_t \leq -1$. Summarizing, we have $d''_t<0$ for all $t\in \deactivating$.

For $t\in \deactivating'$, $d'_t<0$ and $d_t = 0$ holds by definition due to $\deactivating'\subseteq A-\deactivating$, thus $d''_t = d_t'+\delta d_t = d_t' < 0$ which together with the previous paragraph yields condition~(a).

Due to $\CSP-\deactivating \supseteq (\CSP-\deactivating)-\deactivating'=\CSP-(\deactivating\cup \deactivating')$, for any $t\in \CSP-(\deactivating\cup \deactivating')$ we have $d_t = 0$ and $d_t' = 0$, which implies $d''_t = d'+\delta d = 0$, thus verifying condition (b).

Finally, we have $d''\in M^*$ because $d,d'\in M^*$ and $\delta\ge0$.
\end{proof}

Theorem~\ref{th:combining_deactivating_vectors} allows us to combine $\deactivating_i$-deactivating direction~$d^i$ for~$\CSP_i=\CSP_{i-1}- \deactivating_{i-1}$ with $\deactivating_{i-1}$-deactivating direction~$d^{i-1}$ for~$\CSP_{i-1}$ into a single $(\deactivating_{i-1}\cup \deactivating_i)$\nobreakdash-deactivating direction for~$\CSP_{i-1}$. Iteratively, we can thus gradually build a $\bigl(\bigcup_{i=0}^n \deactivating_i\bigr)$\nobreakdash-deactivating direction for~$\CSP$, which certifies unsatisfiability of~$\CSP$ whenever Algorithm~\ref{al:propagation_with_storing_vectors} detects on line~6 that~$A_{n+1}$ (and thus also~$A$) is unsatisfiable.

However, it is not always necessary to construct a full $\bigl(\bigcup_{i=0}^n \deactivating_i\bigr)$-deactivating direction because not every iteration of constraint propagation may have been necessary to prove unsatisfiability of~$\CSP$. Instead, we can use the scope~$\scope\in\constraints$ satisfying~$\CSP_{n+1}\cap\TPA{\scope}=\emptyset$  (where $\CSP_{n+1}=\CSP-\bigcup_{i=0}^n \deactivating_i$, as mentioned above) returned by Algorithm~\ref{al:propagation_with_storing_vectors} on line~7 and construct an $\Mset$\nobreakdash-deactivating direction~$d^*$ for a (usually smaller) set $\Mset \subseteq \bigcup_{i=0}^n \deactivating_i$ such that $(\CSP-\Mset)\cap\TPA{\scope}=\emptyset$. Such a direction~$d^*$ still certifies unsatisfiability of~$\CSP$ and can be sparser and/or may have lower objective values $\VAL {d^*}x$ than a $\bigl(\bigcup_{i=0}^n \deactivating_i\bigr)$\nobreakdash-deactivating direction, which is desirable as explained in~\S\ref{se:properties_of_the_method}.

This is outlined in Algorithm~\ref{al:composition_of_some_vectors}, which composes only a subsequence of directions~$d^i$ based on a given set of indices~$\chosenindices \subseteq \{0,\dots,n\}$ and constructs an $\Mset$-deactivating direction with $\Mset \supseteq \bigcup_{i \in \chosenindices}\deactivating_i$. Although Algorithm~\ref{al:composition_of_some_vectors} is applicable for any set~$\chosenindices$, in our case $\chosenindices$ is obtained by taking a scope $\scope\in\constraints$ such that $\CSP_{n+1}\cap \TPA{\scope}=\emptyset$ and then setting
\begin{equation}
\chosenindices = \ourset{i\in\{0,\dots,n\} \mid \deactivating_i\cap \TPA{\scope}\neq \emptyset}
\label{eq:chosenindices}
\end{equation}
so that $(\CSP-\Mset)\cap\TPA{\scope} = \emptyset$ due to the following fact:

\begin{proposition}
Let $\scope\in\constraints$ be such that $(\CSP-\bigcup_{i=0}^n\deactivating_i)\cap \TPA{\scope}=\emptyset$.
Let $\chosenindices$ be given by~\eqref{eq:chosenindices}.
Then $(\CSP-\bigcup_{i\in I}\deactivating_i)\cap \TPA{\scope}=\emptyset$.
\end{proposition}
\begin{proof}
For any sets $\CSP,\deactivating,T'\subseteq T$ we have $(\CSP-\deactivating)\cap T'=(T'-\deactivating)\cap \CSP$. In particular, $(\CSP-\bigcup_{i=0}^n\deactivating_i)\cap \TPA{\scope} = (\TPA{\scope} - \bigcup_{i=0}^n\deactivating_i) \cap \CSP$. But
$\TPA{\scope} - \bigcup_{i=0}^n\deactivating_i
= \TPA{\scope} - \bigcup_{i\in I}\deactivating_i$
because for each $i\notin I$ we have $\deactivating_i\cap\TPA{\scope}=\emptyset$ which is equivalent to $\TPA{\scope}-\deactivating_i=\TPA{\scope}$.
\end{proof}

\begin{algorithm}[t]
\caption{The procedure {\tt compose} takes the sequences~$(\deactivating_i)_{i=0}^n$ and~$(d^i)_{i=0}^n$ (generated by the procedure {\tt propagate}) and a non-empty index set~$\chosenindices \subseteq \{0,\dots,n\}$ and composes them to an $\Mset$-deactivating direction~$d^*$ for~$\CSP$.} \label{al:composition_of_some_vectors}
\begin{algorithmic}[1]
\State \textbf{procedure} $(\Mset,d^*)=\text{\tt compose}((\deactivating_i)_{i=0}^n,(d^i)_{i=0}^n,\chosenindices)$
\State Initialize $i := \max \chosenindices$, $d^* := d^i$, $\Mset := \deactivating_i$.
\While{$i>0$}
    \State $i:=i-1$
    \If{\normalfont $i\in \chosenindices$ or $\exists t \in \deactivating_i\colon d^*_t\neq 0$}
    \State $d^*:=d^*+\delta d^i$ (where $\delta$ is~\eqref{eq:combining_delta} with $d',d,\deactivating$ replaced by $d^*,d^i,\deactivating_i$)
    \State $\Mset:=\Mset\cup \deactivating_i$
    \EndIf
\EndWhile
\State \Return{$(\Mset,d^*)$}
\end{algorithmic}
\end{algorithm}

Correctness of Algorithm~\ref{al:composition_of_some_vectors} is given by the following theorem:

\begin{proposition}\label{th:composition_of_some_vectors}
Algorithm~\ref{al:composition_of_some_vectors} returns an $\Mset$-deactivating direction~$d^*$ for~$\CSP$ where $\bigcup_{i\in \chosenindices}\deactivating_i \subseteq \Mset  \subseteq \bigcup_{i=0}^n \deactivating_i$.
\end{proposition}
\begin{proof}
The fact that $\Mset \supseteq \bigcup_{i\in \chosenindices}\deactivating_i$ is obvious due to $\deactivating_{\max \chosenindices}\subseteq \Mset$ by initialization on line~2 and $\deactivating_i\subseteq \Mset$ for any $i\in \chosenindices$ such that $i<\max \chosenindices$ because in such case the update on line~7 is performed. Similarly, $\Mset\subseteq \bigcup_{i=0}^n \deactivating_i$ holds by initialization of~$\Mset$ on line~2 and updates on line~7.

It remains to show that $d^*$ is $\Mset$-deactivating, which we will do by induction. We claim that vector~$d^*$ is always $\Mset$-deactivating direction for~$\CSP_i$ on line~3 and $\Mset$-deactivating direction for~$\CSP_{i+1}$ on line~5.

Initially, we have $d^*=d^i$, so $d^*$ is $\deactivating_i$-deactivating (i.e., $\Mset$-deactivating since $\Mset=\deactivating_i$ before the loop is entered) for~$\CSP_i$. Also, when vector~$d^*$ is first queried on line~5, $i$~decreased by~1 due to the update on line~4, so $d^*$~is $\Mset$-deactivating for~$\CSP_{i+1}$. The required property thus holds when the condition on line~5 is first queried with $i=\max \chosenindices-1$.

We proceed with the inductive step. If the condition on line~5 is not satisfied, then necessarily $d^*_t=0$ for all $t\in \deactivating_i$. So, if $d^*$ is $\Mset$-deactivating for~$\CSP_{i+1}$, then it is also $\Mset$-deactivating for~$\CSP_i=\CSP_{i+1}\cup \deactivating_i$, as seen from Definition~\ref{de:t-deactivating}.

If the condition on line~5 is satisfied, $d^*$ is $\Mset$-deactivating for~$\CSP_{i+1}$ before the update on lines~6-7. Since $\CSP_{i+1}=\CSP_i-\deactivating_i$ and $d^i$ is $\deactivating_i$-deactivating for~$\CSP_i$, Theorem~\ref{th:combining_deactivating_vectors} can be applied to~$d^i$ and~$d^*$ to obtain an $(\Mset\cup \deactivating_i)$-deactivating direction for~$\CSP_i$. After updating~$\Mset$ on line~7, it becomes $\Mset$-deactivating for~$\CSP_i$.

When eventually $i=0$, $d^*$~is $\Mset$-deactivating for~$\CSP_0=\CSP$ by line~2 in Algorithm~\ref{al:propagation_with_storing_vectors}.
\end{proof}

\begin{remark}
This is similar to what the VAC~\cite{Cooper-AI-2010} or Augmenting DAG algorithm~\cite{Koval76,Werner-PAMI07} do for arc consistency. To attempt to disprove satisfiability of CSP~$A^*(f)$, these algorithms enforce AC of~$A^*(f)$, during which the causes for forbidding tuples are stored. If the empty AC closure of~$A^*(f)$ is detected (which corresponds to $\TPA{\scope}\cap \CSP_{n+1}=\emptyset$ for some $\scope\in\constraints$), these algorithms do not iterate through all previously forbidden tuples but only trace back the causes for forbidding the elements of the wiped-out domain (here, the elements of~$\TPA{\scope}$).
\end{remark}

\subsection{Line Search}\label{se:line_search}

In~\S\ref{se:improving_directions} we showed how to construct an $\deactivating$-deactivating direction~$d$ for a CSP~$\CSP$, which certifies unsatisfiability of~$\CSP$ whenever $(\CSP-\deactivating)\cap \TPA{\scope}=\emptyset$ for some $\scope\in\constraints$. Given a WCSP $f\in\R^T$ with $A^*(f)=A$, to obtain $f'\in f+M^*$ with $\UB(f')<\UB(f)$ (as in Theorem~\ref{th:CSPcert}), we need to find a step size $\alpha>0$ so that $f'=f+\alpha d$, as discussed in~\S\ref{se:constraint_propagation_scheme}. That means, we need to find $\alpha>0$ such that $\UB(f+\alpha d)<\UB(f)$. This task is known in numerical optimization as \emph{line search}.

Finding the best step size (i.e., exact line search) would require finding a global minimum of the univariate convex piecewise-affine function $\alpha\mapsto \UB(f+\alpha d)$. As this would be too expensive for large WCSP instances, we find only a suboptimal step size (approximate line search) in the following theorem.\footnote{In detail, the step size $\min\{\beta,\gamma\}$ computed in Theorem~\ref{th:improve_bound} corresponds to the first break (i.e., non-differentiable) point of the univariate function with a lower objective. This is analogous to the \emph{first-hit strategy\/} in~\cite[\S3.1.4]{dlask2018minimizing}.}

\begin{theorem}\label{th:improve_bound}
Let $f \in \R^T$. Let $d$ be an $\deactivating$-deactivating direction for~$A^*(f)$. Denote\,\footnote{$\beta$ is always defined: by Definition~\ref{de:t-deactivating} we have $\VAL dx\ge 0$ for all~$x$, hence $\exists t\colon d_t<0\Rightarrow\exists t'\colon d_{t'}>0$. $\gamma$~is defined and needed only in~(c), where we assume that $(A^*(f)-\deactivating)\cap \TPA{\scope}=\emptyset$ for some $\scope\in\constraints$. If the set in the definition of~$\gamma$ is empty, then we define $\gamma=+\infty$ and thus $\min\{\beta,\gamma\}=\beta$.}
\begin{align*}
\beta &= \min\ourset[\bigg]{\frac{\max\nolimits_{t\in \TPA{\scope'}}f_t-f_{t'}}{d_{t'}} \mid \scope'\in\constraints, \; t'\in\TPA{\scope'}, \; d_{t'} > 0 }, \\
\gamma &= \min\bigg\{\,\frac{f_t-f_{t'}}{d_{t'}-d_t} \;\bigg\vert\; \scope\in\constraints, \; (A^*(f)- \deactivating)\cap \TPA{\scope}=\emptyset, \bigg.\\
\bigg. & \hphantom{= \min\bigg\{\,\frac{f_t-f_{t'}}{d_{t'}-d_t} \;\bigg\vert\; } \; t \in \TPA{\scope}\cap \deactivating, \; t'\in \TPA{\scope}-\deactivating, \; d_{t'} > d_t\,\bigg\} .
\end{align*}
Then $\beta,\gamma>0$ and for every $\scope\in\constraints$ and $\alpha\in\R$, WCSP $f'= f + \alpha d$ satisfies:
\begin{enumerate}
\item[(a)] If $(A^*(f)-\deactivating)\cap \TPA{\scope}\neq\emptyset$ and $0\leq \alpha\leq\beta$, then $\max_{t \in \TPA{\scope}}f'_t = \max_{t \in \TPA{\scope}}f_t$.
\item[(b)] If $(A^*(f)-\deactivating)\cap \TPA{\scope}\neq\emptyset$ and $0<\alpha<\beta$, then $A^*(f')\cap \TPA{\scope} = (A^*(f)-\deactivating)\cap \TPA{\scope}$.
\item[(c)] If $(A^*(f)-\deactivating)\cap \TPA{\scope}=\emptyset$ and $0<\alpha\leq\min\{\beta,\gamma\}$, then $\max_{t \in \TPA{\scope}}f'_t < \max_{t \in \TPA{\scope}}f_t$.
\end{enumerate}
\end{theorem}

\begin{proof}
We have $\beta>0$ because $d_{t'} > 0$ implies $t'$~is an inactive tuple, so $\max_{t \in \TPA{\scope}}f_t>f_{t'}$. We have $\gamma>0$ because in $f_t-f_{t'}$ tuple~$t$ is always active and $t'$~is inactive, hence $f_t > f_{t'}$.

To prove (a), let $t^*\in (A^*(f)-\deactivating)\cap \TPA{\scope}$. Hence, by Definition~\ref{de:t-deactivating}, $d_{t^*} = 0$ and the value $\max_{t \in \TPA{\scope}}f'_t$ does not decrease for any $\alpha$ since $f'_{t^*}=f_{t^*}+\alpha d_{t^*}=f_{t^*}$. To show the maximum does not increase, consider a tuple $t' \in \TPA{\scope}$ such that $d_{t'}>0$ (due to $\alpha \geq 0$, tuples with $d_{t'}\leq0$ cannot increase the maximum). It follows that $\alpha\leq \beta\le \tfrac{\max_{t \in \TPA{\scope}}f_t-f_{t'}}{d_{t'}}$, so $f'_{t'} = f_{t'}+d_{t'}\alpha \leq \max_{t \in \TPA{\scope}} f_t$.

To prove~(b), let $(A^*(f)-\deactivating)\cap \TPA{\scope}\neq\emptyset$. As in~(a), we have $\max_{t\in \TPA{\scope}}f_t=\max_{t\in \TPA{\scope}}f_t'$. If $t \in (A^*(f)-\deactivating)\cap \TPA{\scope}$, then $d_t = 0$ and such tuples remain active by $f_t'=f_t$. Tuples $t \in \deactivating\cap \TPA{\scope}$ become inactive since $f'_t = f_t+d_t\alpha<f_t = \max_{t' \in \TPA{\scope}}f_{t'}$ by $d_t < 0$ and $\alpha > 0$. Tuples $t\notin A^*(f)$ either satisfy $d_t\leq 0$ and cannot become active or satisfy $d_t> 0$ and by $\alpha < \beta\le \tfrac{\max_{t'\in \TPA{\scope}}f_{t'}-f_t}{d_t}$, $f'_t = f_t+d_t\alpha < \max_{t' \in \TPA{\scope}} f_{t'}$, so $t\notin A^*(f')$.

To prove (c), let $(A^*(f)-\deactivating)\cap \TPA{\scope}=\emptyset$. For all $t \in \TPA{\scope}\cap \deactivating$, we have $f'_t = f_t+\alpha d_t<f_t$  by~$d_t < 0$ and $\alpha > 0$, i.e., $\max_{t \in \TPA{\scope}\cap \deactivating} f'_t < \max_{t \in \TPA{\scope}\cap \deactivating} f_t$. We proceed to show that $f_t'\leq\max_{t' \in \TPA{\scope}\cap \deactivating} f'_{t'}$ for every $t'\in \TPA{\scope}-\deactivating$. Let $t^*\in \TPA{\scope}\cap \deactivating$ satisfy $f_{t^*}'=\max_{t \in \TPA{\scope}\cap \deactivating} f'_t$. If $d_{t'}>d_{t^*}$, $\alpha \leq \gamma \leq \tfrac{f_{t^*}-f_{t'}}{d_{t'}-d_{t^*}}$ implies $f'_{t^*} = f_{t^*} + \alpha d_{t^*} \geq f_{t'}+\alpha d_{t'} = f'_{t'}$. If $d_{t'} \leq d_{t^*}$, then also $\alpha d_{t'}\leq \alpha d_{t^*}$ and $f'_{t'} = f_{t'} + \alpha d_{t'} \leq f_{t^*}+\alpha d_{t^*} = f'_{t^*}$ holds for any $\alpha \geq 0$ since $f_{t'}< f_{t^*}$. As a result, $\max_{t' \in \TPA{\scope}-\deactivating} f'_{t'} \leq \max_{t \in \TPA{\scope}\cap \deactivating}f_t' < \max_{t \in \TPA{\scope}\cap \deactivating} f_t = \max_{t\in \TPA{\scope}} f_t$.
\end{proof}

If $d$~is an $\deactivating$-deactivating direction for CSP~$A^*(f)$ and for all $\scope\in\constraints$ we have $(A^*(f)-\deactivating)\cap \TPA{\scope}\neq\emptyset$ then, by Theorem~\ref{th:improve_bound}(a,b), there is $\alpha>0$ such that $f'=f+\alpha d$ satisfies $\UB(f')=\UB(f)$ and $A^*(f')=A^*(f)-\deactivating$. This justifies why such direction~$d$ is called $\deactivating$-deactivating: a suitable update of~$f$ along this direction makes tuples~$\deactivating$ inactive for~$f$.

\begin{remark}
This might suggest that to improve the current bound~$\UB(f)$, we need not use Algorithm~\ref{al:composition_of_some_vectors} to construct an $\Mset$-deactivating direction~$d^*$ with $(A^*(f)-\Mset)\cap \TPA{\scope}=\emptyset$ for some $\scope \in \constraints$, but instead, perform steps using the intermediate $\deactivating_i$-deactivating directions~$d^i$ to create a sequence $f^{i+1}=f^i+\alpha_i d^i$ satisfying~$\UB(f^0)=\UB(f^1)=\cdots=\UB(f^n)>\UB(f^{n+1})$. Unfortunately, it is hard to make this work reliably as there are many choices for the intermediate step sizes $0<\alpha_i<\beta_i$. We empirically found Algorithm~\ref{al:composed_algorithm} to be preferable. 
\end{remark}

If $d$~is an $\deactivating$-deactivating direction for~$A^*(f)$ and for some $\scope\in\constraints$ we have $(A^*(f)-\deactivating)\cap \TPA{\scope}=\emptyset$, then, by Theorem~\ref{th:improve_bound}(a,c), there is $\alpha>0$ such that $f'=f+\alpha d$ satisfies $\UB(f')<\UB(f)$. The following corollary of Theorem~\ref{th:improve_bound} finally justifies why the certificate~$d$ of unsatisfiability of CSP~$\CSP^*(f)$ is an improving direction for~\eqref{eq:LP}:

\begin{corollary}
CSP $\CSP\subseteq T$ is unsatisfiable if and only if there is $d\in M^*$ such that for every $f\in\R^T$ with $\CSP=A^*(f)$ there exists $\alpha>0$ such that $\UB(f+\alpha d)<\UB(f)$.
\end{corollary}
\begin{proof}
First, if for some~$\scope\in\constraints$ we have that $\CSP\cap \TPA{\scope}=\emptyset$, $\CSP$ is unsatisfiable and no $f\in\R^T$ satisfies $\CSP=A^*(f)$, so the second condition is trivially satisfied by choosing any $d\in M^*$.

Otherwise, let $d$ be any $\CSP$-deactivating direction (which exists by Theorem~\ref{th:deact_exists}). It follows from Theorem~\ref{th:improve_bound} that for any~$f\in\mathbb{R}^T$ with~$A^*(f)=\CSP$, we can compute a suitable step size~$\alpha > 0$ such that~$\UB(f+\alpha d)<\UB(f)$. The remaining part follows from Theorem~\ref{th:CSPcert}.
\end{proof}

\subsection{Final Algorithm}\label{se:final_alg}

\begin{algorithm}[t]
\caption{The final algorithm to iteratively improve feasible solutions to~\eqref{eq:LP}.} \label{al:composed_algorithm}
\hspace*{\algorithmicindent}\textbf{input:} {WCSP~$g\in\mathbb{R}^T$}
\begin{algorithmic}[1]
\State Initialize $f := g$.
\Repeat
  \State $(\scope,(\deactivating_i)_{i=0}^n,(d^i)_{i=0}^n):=\text{\tt propagate}(A^*(f))$
  \If{$\scope\neq\emptyset$}
     \State Define $\chosenindices$ as in~\eqref{eq:chosenindices}.
     \State $(\Mset,d^*):=\text{\tt compose}((\deactivating_i)_{i=0}^n,(d^i)_{i=0}^n,\chosenindices)$
     \State {Update $f := f+\min\{\beta,\gamma\} d^*$ following Theorem~\ref{th:improve_bound}.}
  \EndIf
\Until{$\scope=\emptyset$}
\State \Return{{$\UB(f)$}}
\end{algorithmic}
\end{algorithm}

Having certificates of unsatisfiability from~\S\ref{se:improving_directions} and step sizes from~\S\ref{se:line_search}, we can now formulate in detail the iterative method outlined in~\S\ref{se:constraint_propagation_scheme}, see Algorithm~\ref{al:composed_algorithm}. First, constraint propagation is applied to CSP~$A^*(f)$ by Algorithm~\ref{al:propagation_with_storing_vectors} until either $A^*(f)$ is proved unsatisfiable or no more propagation is possible. In the latter case, the algorithm halts and returns~$\UB(f)$ as the best achieved upper bound on the optimal value of WCSP~$g$. Otherwise, if $A^*(f)$ is proved unsatisfiable due to $A_{n+1}\cap \TPA{\scope}=\emptyset$ for some $\scope\in \constraints$, define~$\chosenindices$ as in~\eqref{eq:chosenindices} so that $(A^*(f)-\bigcup_{i\in \chosenindices} \deactivating_i)\cap \TPA{\scope}=\emptyset$, and compute an $\Mset$-deactivating direction~$d^*$ where $\Mset\supseteq \bigcup_{i\in \chosenindices} \deactivating_i$ using Proposition~\ref{th:composition_of_some_vectors}. Since $(A^*(f)-\Mset)\cap \TPA{\scope}=\emptyset$, we can update WCSP~$f$ using Theorem~\ref{th:improve_bound}. Consequently,  the bound~$\UB(f)$ strictly improves after each update on line~7.

\begin{remark}
In the maximization version of WCSP, hard constraints can be modelled by allowing minus-infinite weights, i.e., we then have~$g\in(\mathbb{R}\cup\{-\infty\})^T$. We argue that Algorithm~\ref{al:composed_algorithm} can be easily extended to such a setting. Without loss of generality, one can assume that
\begin{equation}\label{eq:assumption_infinite}
    \forall \scope\in\constraints \,\exists t\in T_S\colon g_t\in\mathbb{R},
\end{equation}
i.e., there is at least one finite weight in each scope for the input WCSP~$g$ (as otherwise the WCSP is infeasible).

With this assumption, the definition of the active-tuple CSP~$A^*(\cdot)$ remains unchanged and the tuples with minus-infinite weights are never active. Next, see that propagation in the active-tuple CSP and construction of the improving direction depend only on~$A^*(f)$, so these subroutines need not be modified and, consequently, the improving direction~$d^*$ still contains only finite weights, i.e., $d^*\in\mathbb{R}^T$.

The only difference can arise when computing the step size~$\alpha=\min\{\beta,\gamma\}$ by Theorem~\ref{th:improve_bound}. If~$f\in\mathbb{R}^T$, then~$\alpha$ is always finite. In contrast, if~$f\in(\mathbb{R}\cup\{-\infty\})^T$, then it may happen that $\beta=\gamma=\infty$, so~$\alpha=\min\{\beta,\gamma\}=\infty$ where we assume the usual arithmetic with infinities, so, e.g., $a-(-\infty)=\infty$ for~$a\in\mathbb{R}$. As discussed earlier, the weights of the active tuples are always finite, which avoids indeterminate expressions when computing~$\beta$ and~$\gamma$. Note, arithmetic with infinities is different from the addition with ceiling operator~\cite{Cooper-AI-2010,larrosa2004solving} (unless the ceiling is infinite).

Next, we comment on the finite and infinite case:
\begin{itemize}
    \item If the computed step size is finite, then $\alpha d^*\in\mathbb{R}^T$ and the update on line~7 can be performed following the aforementioned arithmetic with infinities. In this case, condition~\eqref{eq:assumption_infinite} holds for the updated~$f$ since the set of tuples with minus-infinite weight (i.e., the set~$\{t\in T\mid f_t=-\infty\}$) is kept unchanged by the update and we can continue with the next iteration.
    \item On the other hand, if the step size is infinite, then the bound~$B(f+\alpha d^*)$ can be made arbitrarily low by setting~$\alpha$ large enough. Stated formally, this means $\forall b\in \mathbb{R}\,\exists \alpha>0\colon B(f+\alpha d^*)\leq b$, which proves infeasibility of the WCSP instance, so the algorithm should return~$-\infty$ and terminate.
\end{itemize}

All in all, if hard constraints are allowed, the only required change in Algorithm~\ref{al:composed_algorithm} is that, if~$\beta=\gamma=\infty$ on line~7, then the algorithm should terminate and return~$-\infty$ (which is an upper bound on the infeasible initial WCSP).
\end{remark}

In Algorithm~\ref{al:composed_algorithm} we additionally used a heuristic analogous to \emph{capacity scaling\/} in network flow algorithms~\cite[\S7.3]{ahuja1993network}. On line~3 of Algorithm~\ref{al:composed_algorithm}, we replace the active tuples~$A^*(f)$ with `almost' active tuples
\begin{equation}\label{eq:theta-active}
    A^*_\theta(f) = \ourset[\Big]{t=(\scope,k) \in T \mid f_t \geq \max_{t' \in \TPA{\scope}}f_{t'}- \theta}
\end{equation}
for some threshold $\theta>0$.\footnote{This is similar to the notion of Bool$_\theta(f)$ in~\cite[\S11.1]{Cooper-AI-2010}, tolerance~$\delta$ in~\cite[\S4.2]{werner-dlask-maxsat}, and mi$_\epsilon[f]$ in~\cite[\S6.2.4]{Bogdan-book-2019}.} This forces the algorithm to disprove satisfiability using tuples that are far from being active, thus hopefully leading to larger step sizes and faster decrease of the bound. Initially, $\theta$~is set to a high value and whenever we are unable to disprove satisfiability of~$A^*_\theta(f)$,  the current~$\theta$ is decreased as $\theta := \theta/10$. The process continues until $\theta$ becomes very small.

Although our theoretical results are more general, our implementation is limited only to binary WCSPs, i.e., instances where the maximum arity of the weighted constraints is at most 2. We implemented two versions of Algorithm~\ref{al:composed_algorithm} (including capacity scaling\footnote{\label{foo:decreasing}In detail, we initialized $\theta=\max_{k_i,k_j}g_{\{i,j\}}(k_i,k_j)-\min_{k_i,k_j}g_{\{i,j\}}(k_i,k_j)+\max_{k} g_{i'}(k)-\min_k g_{i'}(k)$ where $\{i,j\}\in \constraints$ and $i'\in \vars$ is the edge and variable with the lowest index (based on indexing in the input instance), respectively. The terminating condition was~$\theta\leq10^{-6}$. Let us note that if capacity scaling is used with~$\theta>0$ and the construction of the improving direction is deterministic (which is the case in our implementation), then the method is guaranteed to terminate after a finite number of iterations. This follows from our more general results that we state in~\cite[\S2.2.1]{dlask-thesis}. In order to improve the efficiency of our method, we also decreased~$\theta$ whenever the bound did not improve by more than~$10^{-15}$ in 20 consecutive iterations.}), differing in the local consistency used to attempt to disprove satisfiability of CSP $\CSP^*(f)$:
\begin{itemize}
    \item \emph{Virtual singleton arc consistency via super-reparametrizations (VSAC-SR)}
    uses singleton arc consistency. Precisely, we alternate between AC and SAC propagators: whenever a single tuple $(i,k)\in \vars\times \domain$ is removed by SAC, we step back to enforcing AC until no more AC propagations are possible, and repeat.
    \item \emph{Virtual cycle consistency via super-reparametrizations (VCC-SR)} is the same as VSAC-SR except that SAC is replaced by CC. Though our implementation is different than~\cite{komodakis2008beyond} (we compose deactivating directions rather than alternate between the cycle-repair procedure and the Augmenting DAG algorithm), it has the same fixed points.
\end{itemize}
The procedures for generating deactivating directions for AC, SAC and CC were implemented as described in Examples~\ref{ex:ac}, \ref{ex:sac}, and~\ref{ex:cc}. We used AC3 algorithm to enforce AC. In SAC and CC it is useful to step back to AC whenever possible because deactivating directions of AC correspond to reparametrizations rather than super-reparametrizations, which is desirable as explained in~\S\ref{se:properties_of_the_method}.

\begin{remark}\label{re:virtual_consistencies}
In analogy to~\cite{Cooper-AI-2010,nguyen2017triangle}, let us call a WCSP instance~$f$ \emph{virtual $\Phi$-consistent} (e.g., virtual AC or virtual RPC) if $A^*(f)$~has a non-empty $\Phi$-consistency closure. Then, a \emph{virtual $\Phi$-consistency algorithm} naturally refers to an algorithm to transform a given WCSP instance to a virtual $\Phi$-consistent WCSP instance. In the VAC algorithm, this transformation is equivalence-preserving, i.e., a reparametrization. But in our case, it is a super-reparametrization, which is why we call our algorithms VSAC-SR and VCC-SR.
\end{remark}

Since we restricted ourselves to binary WCSPs, let $\edges=\ourset{\scope\in\constraints\mid |\scope|=2}$ so that $(\vars,\edges)$ is an undirected graph. The cycles in VCC-SR were chosen as follows: if $2|\edges|/|\vars|\leq 5$ (i.e., the average degree of the nodes in~$(\vars,\edges)$ is at most~5), then all cycles of length 3 and 4 present in the graph $(\vars,\edges)$ are used. If $2|\edges|/|\vars|\leq10$, then all cycles of length~3 present in the graph are used. If $2|\edges|/|\vars|>10$ or the above method did not result in any cycles, we use all fundamental cycles w.r.t.\ a spanning tree of the graph~$(\vars,\edges)$.\footnote{Let~$(\vars,\edges')$ be a spanning tree of~$(\vars,\edges)$. A \emph{fundamental cycle} w.r.t. the spanning tree is the unique cycle in the graph~$(\vars,\edges'\cup\{e\})$ where~$e\in \edges-\edges'$. By choosing different edges~$e\in\edges-\edges'$, we obtain the set of all~$|\edges-\edges'|$ fundamental cycles w.r.t. the spanning tree~\cite[Chapter~9]{rosen2000handbook}.} No additional edges are added to the graph. Note, \cite{komodakis2008beyond}~experimented with grid graphs (where cycles of length 4 and 6 of the grid were used) and complete graphs (where cycles of length 3 were used).

Since both VSAC-SR and VCC-SR start by enforcing VAC (i.e., making $A^*(f)$ arc consistent by reparametrizations), before running these methods we used toulbar2 to reparametrize the input WCSP instance to a VAC state (because a specialized algorithm is faster than the more general Algorithm~\ref{al:composed_algorithm}). We employed specialized data structures for storing the sequences~$(\deactivating_i)_{i=0}^n$ and~$(d^i)_{i=0}^n$ from Algorithm~\ref{al:propagation_with_storing_vectors}, which utilize the property that the sets~$(\deactivating_i)_{i=0}^n$ are disjoint and make easier sequential querying of (sparse) vectors~$(d^i)_{i=0}^n$ in Algorithm~\ref{al:composition_of_some_vectors}. Note that the sequence~$(A_i)_{i=0}^{n+1}$ need not be stored and is only needed for theoretical analysis. Moreover, sparse representations were used when composing deactivating directions in Algorithm~\ref{al:composition_of_some_vectors}. To avoid working with `structured' tuples~\eqref{eq:tuples}, we employed a bijection between $T$ and $\{1, \ldots, |T|\}$ to work with numerical indices instead.

Besides the above improvements, we did not fine-tune our implementation for efficiency. Thus, the set~$A^*(f)$ was always calculated by iterating through all tuples (which could be made faster if sparsity of the improving direction was taken into account). The hyper-parameters of our algorithm (e.g., the decrease schedule of~$\theta$ or constants mentioned in Footnote~\ref{foo:decreasing}) were not learned nor systematically optimized. SAC was checked on all active tuples without warm-starting or using any faster SAC algorithm than SAC1~\cite{bessiere2011efficient,Debruyne97somepracticable}. Perhaps most importantly, we did not implement inter-iteration warm-starting as in~\cite{Werner-TR-2005,dlask2018minimizing}, i.e., after updating the weights on line~6 of Algorithm~\ref{al:composed_algorithm}, some deactivating directions in the sequence that were not used to compose the improving direction may be preserved for the next iteration instead of being computed from scratch. Except for computing deactivating directions, the code was the same for VSAC-SR and VCC-SR. We implemented everything in Java.

\subsection{Experiments}\label{se:experiments}

\afterpage{%
    \clearpage
    \begin{landscape}
        \centering 
        \begin{table}
        \begin{footnotesize}
    \caption{Results on instances from Cost Function Library: Average normalized bounds (for each instance group, the best average normalized bound is in bold).}
    \label{tab:cfl-processed}
\resizebox{1.5\textwidth}{!}{
    \begin{tabular}{l|c|l|l|l|l|l|l|l|l|l}
Instance Group & \rotatebox{0}{Instances}& \rotatebox{0}{EDAC}& \rotatebox{0}{VAC}& \rotatebox{0}{VSAC-SR}& \rotatebox{0}{VCC-SR}& \rotatebox{0}{Pseudo-tr.}& \rotatebox{0}{PIC}& \rotatebox{0}{EDPIC}& \rotatebox{0}{maxRPC}& \rotatebox{0}{EDmaxRPC}\\ \hline
/biqmaclib/ & 157  & 0.02 & 0.11 & 0.90 & 0.22 & \textbf{0.92} & 0.83 & 0.81 & 0.79 & 0.81 \\
/crafted/academics/ & 8  & 0.88 & 0.88 & 0.97 & 0.95 & 0.88 & 0.88 & 0.88 & 0.88 & \textbf{1.00} \\
/crafted/auction/paths/ & 420  & 0.00 & 0.09 & 0.91 & 0.35 & \textbf{0.99} & 0.45 & 0.68 & 0.64 & 0.57 \\
/crafted/auction/regions/ & 411  & 0.00 & 0.05 & \textbf{0.99} & 0.10 & 0.98 & 0.08 & 0.18 & 0.23 & 0.13 \\
/crafted/auction/scheduling/ & 419  & 0.00 & 0.02 & \textbf{1.00} & 0.09 & 0.80 & 0.41 & 0.38 & 0.41 & 0.24 \\
/crafted/coloring/ & 33  & 0.94 & 0.94 & 0.99 & 0.97 & 0.98 & \textbf{1.00} & \textbf{1.00} & \textbf{1.00} & 0.99 \\
/crafted/feedback/ & 6  & 0.00 & 0.00 & 0.54 & 0.58 & 0.71 & 0.49 & 0.53 & 0.51 & \textbf{0.72} \\
/crafted/kbtree/ & 1800  & 0.25 & 0.29 & 0.60 & 0.67 & 0.80 & 0.73 & 0.81 & 0.76 & \textbf{0.89} \\
/crafted/maxclique/dimacs\_maxclique/ & 49  & 0.06 & 0.24 & \textbf{0.98} & 0.39 & 0.87 & 0.39 & 0.50 & 0.51 & 0.55 \\
/crafted/maxcut/spinglass\_maxcut/unweighted/ & 5  & 0.00 & 0.00 & \textbf{1.00} & 0.42 & 0.15 & 0.15 & 0.15 & 0.15 & 0.15 \\
/crafted/maxcut/spinglass\_maxcut/weighted/ & 5  & 0.00 & 0.00 & \textbf{1.00} & 0.38 & 0.17 & 0.17 & 0.17 & 0.17 & 0.17 \\
/crafted/modularity/ & 6  & 0.17 & 0.19 & 0.38 & 0.25 & \textbf{0.99} & 0.96 & 0.94 & 0.96 & 0.97 \\
/crafted/planning/ & 65  & 0.00 & 0.54 & \textbf{0.94} & 0.72 & 0.32 & 0.07 & 0.09 & 0.07 & 0.17 \\
/crafted/sumcoloring/ & 43  & 0.04 & 0.15 & 0.47 & 0.50 & \textbf{0.81} & 0.53 & 0.63 & 0.64 & 0.61 \\
/crafted/warehouses/ & 49  & 0.35 & 0.99 & \textbf{1.00} & 0.99 & 0.35 & 0.42 & 0.42 & 0.42 & 0.42 \\
/qaplib/ & 5  & 0.40 & 0.40 & 0.40 & 0.41 & \textbf{0.99} & 0.97 & 0.97 & 0.98 & 0.97 \\
/qplib/ & 23  & 0.00 & 0.10 & \textbf{0.96} & 0.38 & 0.27 & 0.25 & 0.25 & 0.24 & 0.25 \\
/random/maxcsp/completeloose/ & 50  & \textbf{1.00} & \textbf{1.00} & \textbf{1.00} & \textbf{1.00} & \textbf{1.00} & \textbf{1.00} & \textbf{1.00} & \textbf{1.00} & \textbf{1.00} \\
/random/maxcsp/completetight/ & 50  & 0.00 & 0.12 & 0.57 & 0.72 & 0.88 & 0.94 & \textbf{0.99} & 0.69 & 0.76 \\
/random/maxcsp/denseloose/ & 50  & \textbf{1.00} & \textbf{1.00} & \textbf{1.00} & \textbf{1.00} & \textbf{1.00} & \textbf{1.00} & \textbf{1.00} & \textbf{1.00} & \textbf{1.00} \\
/random/maxcsp/densetight/ & 50  & 0.02 & 0.14 & 0.52 & \textbf{1.00} & 0.68 & 0.48 & 0.49 & 0.52 & 0.60 \\
/random/maxcsp/sparseloose/ & 90  & 0.96 & 0.96 & \textbf{1.00} & 0.96 & 0.96 & 0.96 & 0.96 & 0.96 & 0.96 \\
/random/maxcsp/sparsetight/ & 50  & 0.01 & 0.12 & 0.54 & \textbf{1.00} & 0.64 & 0.40 & 0.40 & 0.43 & 0.51 \\
/random/maxcut/random\_maxcut/ & 400  & 0.00 & 0.00 & 0.77 & 0.13 & 0.95 & 0.98 & 0.98 & 0.97 & \textbf{0.99} \\
/random/mincut/ & 500  & 0.09 & \textbf{1.00} & \textbf{1.00} & \textbf{1.00} & 0.10 & 0.10 & 0.10 & 0.10 & 0.10 \\
/random/randomksat/ & 493  & 0.01 & 0.02 & 0.75 & 0.22 & \textbf{0.95} & 0.91 & 0.89 & 0.86 & 0.87 \\
/random/wqueens/ & 6  & 0.00 & 0.52 & \textbf{0.96} & 0.94 & 0.48 & 0.12 & 0.29 & 0.13 & 0.72 \\
/real/celar/ & 23  & 0.00 & 0.05 & 0.08 & 0.16 & \textbf{0.97} & 0.66 & 0.66 & 0.78 & 0.95 \\
/real/maxclique/protein\_maxclique/ & 1  & 0.00 & 0.00 & \textbf{1.00} & 0.03 & 0.93 & 0.04 & 0.04 & 0.08 & 0.04 \\
/real/spot5/ & 1  & 0.00 & 0.08 & \textbf{1.00} & 0.49 & \textbf{1.00} & 0.74 & 0.66 & 0.41 & 0.74 \\
/real/tagsnp/tagsnp\_r0.5/ & 23  & 0.04 & 0.86 & \textbf{0.95} & 0.86 & 0.31 & 0.31 & 0.33 & 0.29 & 0.46 \\
/real/tagsnp/tagsnp\_r0.8/ & 80  & 0.13 & 0.66 & \textbf{0.91} & 0.68 & 0.29 & 0.39 & 0.38 & 0.33 & 0.47 \\ \hline\hline
\textbf{Average over all groups} & 5371 & 0.20 & 0.36 & \textbf{0.82} & 0.58 & 0.72 & 0.56 & 0.58 & 0.56 & 0.62 \\
\textbf{Average over groups with $\geq 5$ instances} & 5369 & 0.21 & 0.38 & \textbf{0.80} & 0.60 & 0.71 & 0.57 & 0.59 & 0.58 & 0.63 \\
    \end{tabular}
    } 
    \end{footnotesize}
    \end{table}
    \end{landscape}
    \clearpage
}

\afterpage{%
    \clearpage
    \begin{landscape}
        \centering 
        \begin{table}
        \begin{footnotesize}
        \caption{Results on instances from Cost Function Library: Average CPU time in seconds (for each instance group, the shortest average CPU time is in bold).}
    \label{tab:cfl-processed-time}
    \resizebox{1.5\textwidth}{!}{
    \begin{tabular}{l|c|r|r|r|r|r|r|r|r|r}
Instance Group & \rotatebox{0}{Instances}&\rotatebox{0}{EDAC}& \rotatebox{0}{VAC}& \rotatebox{0}{VSAC-SR}& \rotatebox{0}{VCC-SR}& \rotatebox{0}{Pseudo-tr.}& \rotatebox{0}{PIC}& \rotatebox{0}{EDPIC}& \rotatebox{0}{maxRPC}& \rotatebox{0}{EDmaxRPC}\\ \hline
/biqmaclib/ & 157 & \textbf{0.11} & 0.12 & 180.07 & 34.60 & 83.25 & 1240.00 & 1241.29 & 1242.16 & 1271.86 \\
/crafted/academics/ & 8 & \textbf{0.11} & \textbf{0.11} & 28.61 & 1.04 & 29.08 & 121.44 & 120.86 & 108.08 & 104.47 \\
/crafted/auction/paths/ & 420 & \textbf{0.04} & \textbf{0.04} & 1.96 & 0.83 & 1.92 & 0.19 & 0.23 & 0.48 & 0.64 \\
/crafted/auction/regions/ & 411 & \textbf{0.20} & 0.32 & 32.14 & 9.45 & 673.42 & 49.85 & 51.37 & 102.61 & 110.48 \\
/crafted/auction/scheduling/ & 419 & \textbf{0.10} & 0.12 & 16.22 & 2.03 & 49.85 & 26.90 & 26.89 & 32.06 & 32.30 \\
/crafted/coloring/ & 33 & \textbf{0.09} & 0.10 & 4.99 & 1.40 & 0.20 & 545.50 & 545.50 & 545.51 & 545.50 \\
/crafted/feedback/ & 6 & \textbf{0.70} & \textbf{0.70} & 3588.39 & 3600.11 & 11.64 & 1860.89 & 1874.08 & 1875.93 & 1873.07 \\
/crafted/kbtree/ & 1800 & \textbf{0.02} & \textbf{0.02} & 3.13 & 11.25 & 0.10 & 0.04 & 0.05 & 0.06 & 0.07 \\
/crafted/maxclique/dimacs\_maxclique/ & 49 & \textbf{0.71} & 1.32 & 279.08 & 126.90 & 955.60 & 1345.67 & 1342.14 & 1429.73 & 1428.12 \\
/crafted/maxcut/spinglass\_maxcut/unweighted/ & 5 & 0.02 & 0.02 & 0.82 & 0.44 & 0.02 & \textbf{0.01} & \textbf{0.01} & \textbf{0.01} & \textbf{0.01} \\
/crafted/maxcut/spinglass\_maxcut/weighted/ & 5 & 0.02 & 0.02 & 1.09 & 0.53 & 0.02 & \textbf{0.01} & \textbf{0.01} & \textbf{0.01} & \textbf{0.01} \\
/crafted/modularity/ & 6 & \textbf{0.19} & 0.29 & 1023.48 & 127.39 & 66.25 & 706.30 & 783.02 & 741.91 & 1442.57 \\
/crafted/planning/ & 65 & \textbf{0.16} & 0.29 & 638.85 & 60.62 & 7.41 & 0.93 & 0.96 & 2.33 & 4.73 \\
/crafted/sumcoloring/ & 43 & \textbf{1.29} & 1.94 & 727.49 & 963.61 & 255.72 & 1508.37 & 1508.36 & 1509.34 & 1512.68 \\
/crafted/warehouses/ & 49 & 4.10 & 9.48 & 735.80 & 735.83 & \textbf{4.09} & 29.48 & 29.54 & 28.80 & 29.82 \\
/qaplib/ & 5 & \textbf{0.08} & 0.09 & 119.05 & 278.53 & 7.38 & 1448.63 & 1444.95 & 1450.09 & 1449.22 \\
/qplib/ & 23 & \textbf{0.13} & 0.14 & 255.85 & 43.11 & 195.32 & 626.25 & 626.24 & 626.27 & 626.36 \\
/random/maxcsp/completeloose/ & 50 & \textbf{0.06} & \textbf{0.06} & 1.31 & 0.16 & 0.48 & 0.09 & 0.10 & 0.19 & 0.18 \\
/random/maxcsp/completetight/ & 50 & \textbf{0.02} & 0.03 & 6.35 & 12.68 & 0.47 & 0.21 & 0.25 & 0.31 & 0.33 \\
/random/maxcsp/denseloose/ & 50 & \textbf{0.02} & \textbf{0.02} & 166.78 & 0.06 & 0.11 & 0.03 & 0.03 & 0.03 & 0.03 \\
/random/maxcsp/densetight/ & 50 & \textbf{0.02} & \textbf{0.02} & 4.20 & 17.38 & 0.10 & 0.06 & 0.07 & 0.07 & 0.08 \\
/random/maxcsp/sparseloose/ & 90 & \textbf{0.03} & \textbf{0.03} & 611.38 & 0.05 & 0.06 & 0.04 & 0.04 & 0.04 & 0.04 \\
/random/maxcsp/sparsetight/ & 50 & \textbf{0.02} & \textbf{0.02} & 11.00 & 9.74 & 0.06 & 0.04 & 0.05 & 0.05 & 0.05 \\
/random/maxcut/random\_maxcut/ & 400 & \textbf{0.01} & \textbf{0.01} & 0.73 & 0.15 & 0.04 & 0.03 & 0.03 & 0.05 & 0.07 \\
/random/mincut/ & 500 & 1.09 & 2.43 & 14.40 & 86.22 & 1.12 & 0.88 & \textbf{0.87} & \textbf{0.87} & \textbf{0.87} \\
/random/randomksat/ & 493 & \textbf{0.02} & \textbf{0.02} & 3.42 & 0.17 & 0.13 & 0.07 & 0.10 & 0.16 & 0.31 \\
/random/wqueens/ & 6 & \textbf{1.33} & 1.49 & 992.85 & 502.42 & 644.87 & 1800.15 & 1800.20 & 1800.18 & 1800.60 \\
/real/celar/ & 23 & \textbf{0.27} & 0.28 & 1798.51 & 2972.69 & 66.56 & 300.76 & 219.91 & 495.26 & 1066.87 \\
/real/maxclique/protein\_maxclique/ & 1 & \textbf{0.26} & 0.44 & 25.24 & 6.77 & 1196.62 & 114.62 & 114.99 & 215.30 & 220.81 \\
/real/spot5/ & 1 & \textbf{0.01} & \textbf{0.01} & 0.62 & 0.08 & 0.11 & 0.03 & 0.03 & 0.04 & 0.04 \\
/real/tagsnp/tagsnp\_r0.5/ & 23 & \textbf{4.83} & 378.77 & 3338.53 & 2897.83 & 239.38 & 3155.96 & 3148.66 & 3172.58 & 3295.19 \\
/real/tagsnp/tagsnp\_r0.8/ & 80 & \textbf{1.52} & 22.82 & 1239.73 & 858.83 & 90.05 & 195.12 & 206.76 & 359.55 & 409.88 \\ \hline\hline
\textbf{Average over all groups} & 5371 & \textbf{0.55} & 13.17 & 495.38 & 417.59 & 143.17 & 471.21 & 471.49 & 491.88 & 538.35 \\
\textbf{Average over groups with $\geq 5$ instances} & 5369 & \textbf{0.58} & 14.04 & 527.54 & 445.20 & 112.82 & 498.80 & 499.08 & 517.49 & 566.88  \\
 \end{tabular}
 } 
\end{footnotesize}
    \end{table}
    \end{landscape}
    \clearpage
}

We compared the bounds calculated by VSAC-SR and VCC-SR with the bounds provided by EDAC~\cite{de2005existential}, VAC~\cite{Cooper-AI-2010}, pseudo-triangles (option \texttt{-t=8000} in toulbar2, adds up to 8~GB of ternary weight functions), PIC, EDPIC, maxRPC, and EDmaxRPC~\cite{nguyen2017triangle}, which are implemented in toulbar2~\cite{toulbar2}. Our motivation for choosing these local consistencies is as follows: EDAC is the typically chosen local consistency that is maintained during branch-and-bound search. VAC is highly related to our approach and can be used in pre-processing (as it is faster than OSAC which is usually too memory- and time-consuming for practical purposes). Finally, we consider a class of recently proposed triangle-based consistencies~\cite{nguyen2017triangle} that enforce stronger forms of local consistency.

We did the comparison on the Cost Function Library benchmark~\cite{cost-function-library}. Due to limited computation resources, we used only the smallest 16500~instances (out of~18132). Of these, we omitted instances containing weight functions of arity~3 or higher. Moreover, to avoid easy instances, we omitted instances that were solved by VAC without search (i.e., toulbar2 with options \texttt{-A -bt=0} found an optimal solution). We also omitted the \emph{validation\/} instances that are used for testing and debugging. Overall, 5371~instances were left for our comparison.

For each instance and each method, we only calculated the upper bound and did not do any search. For each instance and method, we computed the normalized bound $\frac{B_w-B_m}{B_w-B_b}$ where $B_m$ is the bound computed by the method for the instance and $B_w$ and~$B_b$ is the worst and best bound for the instance among all the methods, respectively. Thus, the best bound\footnote{To avoid numerical precision issues, bounds~$B_m$ within $B_b\pm 10^{-4} B_b$ or $B_b\pm 0.01$ are also normalized to~1. If $B_w=B_b$, then the normalized bounds for all methods are equal to~1 on this instance.} transforms to~1 and the worst bound to~0, i.e., greater is better.

For 26~instances, at least one method was not able to finish in the prespecified 1\nobreakdash-hour CPU-time limit. These timed-out methods were omitted from the calculation of the normalized bounds for these instances. From the point of view of the method, the instance was not incorporated into the average of the normalized bounds of this particular method. We note that implementations of VSAC-SR and VCC-SR provide a bound when terminated at any time, whereas the implementations of the other methods provide a bound only when they are left to finish. Time-out happened 5, 2, 3, 6, and 24 times for pseudo-triangles, PIC, EDPIC, maxRPC, and EDmaxRPC, respectively. This did not affect the results much as there were 5371~instances in total.

The results in Table~\ref{tab:cfl-processed} show that no method is best for all instance groups, instead, each method is suitable for a different group. However, VSAC-SR performed best for most groups and otherwise was often competitive to the other strong consistency methods. VSAC-SR seems particularly good at spinglass\_maxcut~\cite{spinglass}, planning~\cite{Cooper2011} and qplib~\cite{Furini2019} instances. Taking the overall unweighted average of group averages (giving the same importance to each group), VSAC-SR achieved the greatest average value.
We also evaluated the ratio to worst bound, $B_m/B_w$, for instances with $B_w\neq 0$; the results were qualitatively the same: VSAC-SR again achieved the best overall average of~3.93 (or~4.15 if only groups with $\geq5$~instances are considered) compared to second-best pseudo-triangles with~2.71 (or~2.84).

The runtimes (on a laptop with i7-4710MQ processor at 2.5~GHz and 16GB RAM) are reported in Table~\ref{tab:cfl-processed-time}. Again, the results are group-dependent and one can observe that the methods explore different trade-offs between bound quality and runtime. However, the strong consistencies are comparable in terms of runtime on average, except for pseudo-triangles, which is a faster method that however needs significantly more memory.

The code that was used to obtain these results is available at~\url{https://cmp.felk.cvut.cz/~dlaskto2/code/VSAC-SR.zip}.

\section{Additional Properties of Super-Reparametrizations}\label{se:theoretical_properties} 

In this section, we present a more detailed study of properties of WCSPs that are preserved by (possibly optimal) super-reparametrizations. To that end, we first revisit in~\S\ref{se:minimal_csp} the notion of a minimal CSP for a set of assignments. The key result of~\S\ref{se:theoretical_properties} is presented in~\S\ref{se:props_optimal}, where we study the relation of the set of optimal assignments of some WCSP to the set of optimal assignments of its super-reparametrization optimal for~\eqref{eq:LP}, showing that they need not coincide in general. In~\S\ref{se:arbitrary_suprepar}, we give some properties of general (i.e., not necessarily optimal for~\eqref{eq:LP}) super-reparametrizations.

\subsection{Minimal CSP}\label{se:minimal_csp}

Let us ask when for a given set $X\subseteq D^V$ of assignments (i.e., a $|V|$-ary relation over~$\domain$) does there exist $A\subseteq T$ such that $X=\SOL(A)$, i.e., when is~$X$ representable as the solution set of a CSP with a given structure $(D,V,C)$. For that, denote
\begin{equation}\label{eq:minCSP}
\textstyle
\minimal(X) = \bigcap \mathcal{A}^\uparrow(X)
\quad\text{where}\quad
\mathcal{A}^\uparrow(X) = \ourset{ A\subseteq T \mid X\subseteq\SOL(A) } .
\end{equation}
Thus, $\mathcal{A}^\uparrow(X)$ is the set of all CSPs whose solution set includes~$X$ and $\minimal(X)$ is the intersection of these CSPs. We call $\minimal(X)$ the \emph{minimal CSP} for~$X$. For CSPs with only binary relations, this concept was studied in~\cite{montanari1974networks} and~\cite[\S2.3.2]{dechter2003constraint}.

\begin{proposition}\label{pr:sol_intersection}
The map $\SOL$ preserves intersections\,\footnote{\label{foo:isotone}This implies that the map $\SOL$ is isotone, i.e., $A_1\subseteq A_2\subseteq T$ implies $\SOL(A_1)\subseteq\SOL(A_2)$. Although isotony is obvious (clearly, enlarging the set of allowed tuples of a CSP preserves or enlarges its solution set), note that isotony does not imply preserved intersections.
A weaker result than our Proposition~\ref{pr:sol_intersection} is \cite[Theorem~3.2]{montanari1974networks}: in our notation, it says that $\SOL(A_1)=\SOL(A_2)$ implies $\SOL(A_1\cap A_2)=\SOL(A_1)$.}, i.e., for any $A_1,A_2\subseteq T$ we have $\SOL(A_1\cap A_2)=\SOL(A_1)\cap \SOL(A_2)$.
\end{proposition}
\begin{proof}
For any $x\in\domain^\vars$ we have
\begin{align*}
    x\in \SOL(A_1)\cap \SOL(A_2) &\iff x\in \SOL(A_1), \; x\in\SOL(A_2)\\
    &\iff \forall \scope\in\constraints\colon (\scope,x[\scope])\in A_1, \; (\scope,x[\scope])\in A_2\\
    & \iff \forall \scope\in\constraints\colon (\scope,x[\scope])\in A_1 \cap A_2\\
    & \iff x\in \SOL(A_1\cap A_2). \qedhere
\end{align*}
\end{proof}

\begin{proposition}\label{pr:clo_intersection}
For any $X\subseteq D^V$, the set $\mathcal{A}^\uparrow(X)$ is closed under intersections, i.e., for any $A_1,A_2\subseteq T$ we have $A_1,A_2\in \mathcal{A}^\uparrow(X)\implies A_1\cap A_2\in \mathcal{A}^\uparrow(X)$.
\end{proposition}
\begin{proof}
If $X\subseteq\SOL(A_1)$ and $X\subseteq\SOL(A_2)$, then $X\subseteq\SOL(A_1)\cap\SOL(A_2)=\SOL(A_1\cap A_2)$, where the equality holds by Proposition~\ref{pr:sol_intersection}.
\end{proof}

Proposition~\ref{pr:clo_intersection} implies $\minimal(X)\in\mathcal{A}^\uparrow(X)$, i.e., $X\subseteq\SOL(\minimal(X))$. This shows that $\minimal(X)$ is the smallest CSP whose solution set includes~$X$. 
It follows that $X=\SOL(\minimal(X))$ if and only if $X=\SOL(A)$ for some $A\subseteq T$.

The minimal CSP for~$X$ can be equivalently defined in terms of tuples:

\begin{proposition}[\cite{montanari1974networks,dechter2003constraint}]\label{pr:amin_alt}
We have $\minimal(X) = \ourset{(S,k)\in T \mid \exists x\in X\colon x[\scope]=k}$.
\end{proposition}
\begin{proof}
Denote $A'=\ourset{(S,k)\in T \mid \exists x\in X\colon x[\scope]=k}$. By definition of~$A'$ we have $\SOL(A')\supseteq X$, so $A'\in \mathcal{A}^\uparrow(X)$ and $\minimal(X)=\bigcap \mathcal{A}^\uparrow(X)  \subseteq A'$.

It remains to show that $\minimal(X)\supseteq A'$. For contradiction, suppose there is a tuple $(S^*,k^*)\in A'-\minimal(X)$. By definition of~$A'$, there exists $x\in X$ such that $x[S^*]=k^*$. However, since $(S^*,x[S^*])=(S^*,k^*)\notin \minimal(X)$, we have $x\notin \SOL(\minimal(X))$. By $x\in X$, this contradicts $X\subseteq \SOL(\minimal(X))$.
\end{proof}

Recall that a CSP~$\CSP$ is \emph{positively consistent}~\cite{5670022,10.1007/978-3-642-40627-0_15} (called `minimal' in~\cite{montanari1974networks,dechter2003constraint}) if and only if for each $(\scope,k)\in \CSP$ there exists $x\in \SOL(\CSP)$ such that $x[\scope]=k$, i.e., each allowed tuple is used by at least one solution, i.e., no tuple can be forbidden without losing some solutions. Proposition~\ref{pr:amin_alt} shows that $\minimal(X)$ is positively consistent for every $X\subseteq D^V$.

\begin{theorem}\label{th:galois}
For any $X\subseteq D^V$ and $\CSP\subseteq T$, we have 
\begin{equation}
\minimal(X)\subseteq \CSP \quad\iff\quad X \subseteq \SOL(\CSP). \label{eq:galois}
\end{equation}
\end{theorem}
\begin{proof}
If $\minimal(X) \subseteq \CSP$, then by isotony of $\SOL$ we have $\SOL(\minimal(X))\subseteq \SOL(\CSP)$. Since $X\subseteq \SOL(\minimal(X))$, we have $X \subseteq \SOL(\CSP)$.

If $X \subseteq \SOL(\CSP)$, i.e., $\CSP\in\mathcal{A}^\uparrow(X)$, then $\minimal(X)=\bigcap\mathcal{A}^\uparrow(X)\subseteq A$.
\end{proof}

Comparing~\eqref{eq:galois} with~\eqref{eq:minCSP} shows that the set $\mathcal{A}^\uparrow(X)$ is just an interval (w.r.t.\ the partial ordering by inclusion):
\begin{equation}
\mathcal{A}^\uparrow(X)
= \ourset{ A \mid \minimal(X) \subseteq A \subseteq T }
= [\minimal(X),T] .
\label{eq:Ainterval}
\end{equation}

Theorem~\ref{th:galois} further reveals that the maps $\minimal$ and $\SOL$ form a \emph{Galois connection}~\cite{davey2002introduction} between sets~$2^{(D^V)}$ and~$2^T$, partially ordered by inclusion (to our knowledge, we are the first to notice this).
Associated with the Galois connection are the closure operator $\SOL\circ\minimal$ and the dual closure operator\footnote{Recall that an operator is a closure (dual closure) if it is isotone, idempotent and increasing (decreasing), see, e.g.,~\cite[\S1.4]{blyth2005lattices}.} $\minimal\circ\SOL$. We have already seen their meaning:
\begin{itemize}
\item For any CSP $\CSP\subseteq T$, the CSP $\minimal(\SOL(\CSP))$ is the \emph{positive consistency closure}\footnote{The term \emph{consistency closure}, as used in constraint programming~\cite{Bessiere-CP-handbook}, is a dual closure in our notation because taking a consistency closure of a CSP \emph{deletes} some of its allowed tuples (it would be a closure if the CSP was defined by a set of \emph{forbidden} tuples).} of~$\CSP$, i.e., the smallest CSP with the same solution set as~$\CSP$.
A CSP~$\CSP$ is positively consistent if and only if $\minimal(\SOL(\CSP))=\CSP$.
\item For any set of assignments $X\subseteq D^V$, $\SOL(\minimal(X))$ is  the smallest (possibly non-strict) superset of~$X$ which is the solution set of some CSP $A\subseteq T$. We have $X=\SOL(\minimal(X))$ if and only if $X=\SOL(A)$ for some $A\subseteq T$.
\end{itemize}

\begin{remark}
Following~\cite[\S7.27]{davey2002introduction}, it is easy to see in this case that the maps~$\minimal$ and~$\SOL$ are mutually inverse bijections (even order-isomorphisms) if we restrict ourselves only to positively consistent CSPs, i.e., $\ourset{\CSP\subseteq T\mid \SOL(\minimal(\CSP))=\CSP}$ and sets of assignments representable as solution sets of some CSP $\CSP\subseteq T$, i.e., $\ourset{\SOL(\CSP)\mid \CSP\subseteq T}$.
\end{remark}

\subsection{Optimal Assignments from Optimal Super-Reparametrizations}\label{se:props_optimal}

Theorem~\ref{th:optimality_cond} says that the optimal value of~\eqref{eq:LP} coincides with the optimal value $\max_x \VAL gx$ of WCSP~$g$. We now focus on the optimal assignments (rather than optimal value) of WCSP~$g$. For brevity, we will denote the set of all optimal assignments of WCSP~$g$ as
\begin{equation}
\OPT(g)= \argmax_{x\in\domain^\vars}\VAL gx \subseteq D^V .
\end{equation}

\begin{theorem}\label{th:f_optimal}
If $f$~is optimal for~\eqref{eq:LP}, then\footnote{Statement~(c) in Theorem~\ref{th:optimality_cond} is equivalent to $\SOL(A^*(f))\cap W(f-g) \neq \emptyset$ where $W(d) = \ourset{x\in\domain^\vars\mid \VAL dx=0}$, i.e., satisfiability of CSP~$A^*(f)$ with an additional global constraint $x\in W(f-g)$. As a corollary of Theorem~\ref{th:f_optimal}, we have that $\SOL(A^*(f))\cap W(f-g) \neq \emptyset \implies \SOL(A^*(f))\cap W(f-g)=\OPT(g)$ for any super-reparametrization~$f$ of~$g$.}
$\OPT(g) \subseteq \OPT(f) = \SOL(A^*(f))$.
\end{theorem}
\begin{proof}
To show $\OPT(g)\subseteq \OPT(f)$, let $x^*\in\OPT(g)$. By Theorem~\ref{th:optimality_cond}, $\VAL g{x^*} = \UB(f)$. Analogously to the proof of Theorem~\ref{th:optimality_cond}: since $\UB(f)\geq \VAL f{x^*}\geq \VAL g{x^*}$, we have that $\UB(f)= \VAL f{x^*}= \VAL g{x^*}$, thus $x^*$ is optimal for WCSP~$f$.

The equality $\OPT(f) = \SOL(A^*(f))$ follows from~$\UB(f)=\max_{x\in \domain^\vars}\VAL fx=\max_{x\in \domain^\vars}\VAL gx$ and Theorem~\ref{th:ub}.
\end{proof}

\begin{figure}
\centering
\begin{subfigure}[t]{0.4077\textwidth}
\centering
\includegraphics[width=\textwidth]{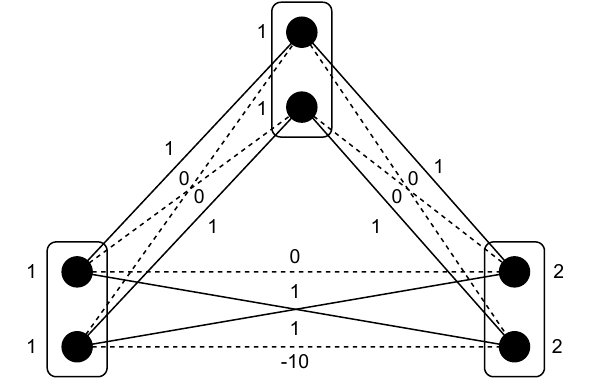}
\caption{WCSP~$g$ with~$\UB(g)=7$.}
\end{subfigure}
\begin{subfigure}[t]{0.4077\textwidth}
\centering
\includegraphics[width=\textwidth]{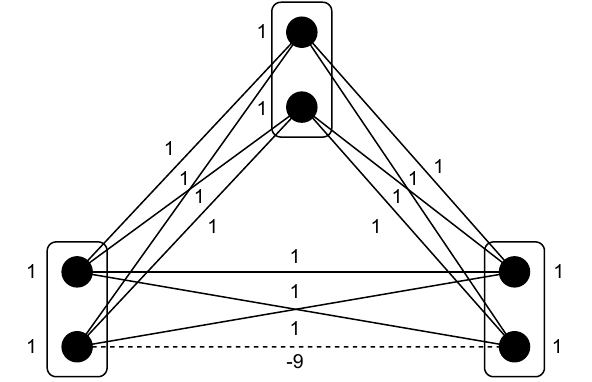}
\caption{WCSP~$f$ with~$\UB(f)=6$.}
\end{subfigure}
\caption{WCSP~$f$ is an optimal super-reparametrization of WCSP~$g$. It is easy to verify that $A^*(f)=\minimal(\OPT(g))$ but $\OPT(f)=\SOL(A^*(f))\supsetneq \OPT(g)$.} 
\label{fig:3cycle}
\end{figure}

Our main goal in~\S\ref{se:theoretical_properties} is to characterize when the inclusion in Theorem~\ref{th:f_optimal} holds with equality, which is given by Theorem~\ref{th:existence_of_CSP} below.

\begin{proposition}\label{pr:A=A*}
For every $g\in\R^T$ and $A\subseteq T$ such that $\OPT(g)\subseteq\SOL(A)$, there exists $f\in\R^T$ optimal for~\eqref{eq:LP} such that $A=A^*(f)$.
\end{proposition}
\begin{proof}
Define the vector~$f$ as
\begin{equation}\label{eq:A=A*}
f_t =
\begin{cases}
F_1/|\constraints| & \text{if $t\in A$} \\
F_2/|\constraints| & \text{if $t\notin A$}
\end{cases}
\quad \forall t \in T
\end{equation}
where
\begin{equation}\label{eq:F12}
F_1 = \max_{x\in \domain^\vars} \VAL gx ,
\quad
F_2 = \max\ourset{\VAL gx \mid x\in \domain^\vars, \VAL gx < F_1}
\end{equation}
are the best and the second-best objective value of WCSP~$g$. Note, if $\OPT(g)=D^V$, then $F_2$~is undefined but it does not matter because it is never used in~\eqref{eq:A=A*}.

Since $\emptyset\neq\OPT(g)\subseteq\SOL(A)$, CSP~$\CSP$ is satisfiable. Therefore for each $\scope \in \constraints$ we have $\CSP\cap\TPA{\scope}\neq\emptyset$, hence
\begin{equation}\label{eq:max_tuple_value}
\max_{t\in \TPA{\scope}} f_t = F_1/|\constraints|.
\end{equation}
Equality $\CSP=A^*(f)$ now follows from~\eqref{eq:A=A*}.

To show that $f$ is feasible for~\eqref{eq:LP}, we distinguish two cases:
\begin{itemize}
\item If $x\in\OPT(g)$, i.e., $\VAL gx=F_1$, then $x\in\SOL(\CSP)=\SOL(\CSP^*(f))$. Therefore for all $\scope\in\constraints$ we have $(\scope,x[\scope])\in A^*(f)$, hence $f_\scope(x[\scope])=F_1/|\constraints|$ by~\eqref{eq:max_tuple_value}. Substituting into~\eqref{eq:energy'} yields $\VAL fx = F_1$. Hence $\VAL fx = F_1 = \VAL gx$.
\item If $x\notin\OPT(g)$, we have $f_t\geq F_2/|\constraints|$ for all $t\in T$, hence $\VAL fx\geq F_2$ by~\eqref{eq:energy'}. By~\eqref{eq:F12} we also have $\VAL gx\le F_2$. Hence $\VAL fx\ge F_2 \ge \VAL gx$.
\end{itemize}

To show that $f$~is optimal for~\eqref{eq:LP}, we use~\eqref{eq:max_tuple_value} to obtain
$\UB(f) = \sum_{\scope\in \constraints} F_1/|\constraints| = F_1 = \max_x\VAL gx$ and apply Theorem~\ref{th:optimality_cond}.
\end{proof}

\begin{theorem}\label{th:TLattice_char}
For every $g\in\R^T$, we have
\begin{equation}
\mathcal{A}^\uparrow(\OPT(g)) = \ourset{ A^*(f) \mid \text{$f$~is optimal for~\eqref{eq:LP}}} .
\label{eq:TLattice_char}
\end{equation}
\end{theorem}
\begin{proof}
The inclusion~$\supseteq$ says that for every optimal~$f$ we have $\OPT(g)\subseteq\SOL(A^*(f))$, which was proved in Theorem~\ref{th:f_optimal}. The inclusion~$\subseteq$ was proved in Proposition~\ref{pr:A=A*}.
\end{proof}

Now we combine the results of~\S\ref{se:minimal_csp} and~\S\ref{se:props_optimal} to obtain the main result of~\S\ref{se:theoretical_properties}. First observe that, by~\eqref{eq:Ainterval}, the set~\eqref{eq:TLattice_char} is just the interval $[\minimal(\OPT(g)),T]$.

\begin{theorem}\label{th:existence_of_CSP}
For every $g\in \R^T$, the following statements are equivalent:
\begin{enumerate}
\item[(a)] $\OPT(g)=\SOL(A)$ for some $A\subseteq T$,
\item[(b)] $\OPT(g) = \OPT(f)$ for some $f$ optimal for~\eqref{eq:LP}.
\end{enumerate}
If both statements are true, then statement~(a) holds, e.g., for $A=\minimal(\OPT(g))$ and statement~(b) holds, e.g., if $A^*(f)=\minimal(\OPT(g))$.
\end{theorem}

\begin{proof}
Let $g\in\R^T$. By Theorem~\ref{th:TLattice_char}, there exists~$f$ optimal for~\eqref{eq:LP} satisfying $\CSP^*(f)=\minimal(\OPT(g))$. By Theorem~\ref{th:f_optimal}, this~$f$ satisfies $\OPT(f)=\SOL(A^*(f))$.

By the results of~\S\ref{se:minimal_csp}, statement~(a) is equivalent to $\OPT(g)=\SOL(\minimal(\OPT(g)))$.
Therefore, if (a)~holds, then (b)~holds for the above~$f$. In the other direction, if (b)~holds for the above~$f$, then (a)~holds.
\end{proof}

Theorem~\ref{th:existence_of_CSP} shows that the inclusion in Theorem~\ref{th:f_optimal} holds with equality for some optimal~$f$ if and only if the set $\OPT(g)$ of optimal assignments of WCSP~$g$ is representable as a solution set of some CSP with the same structure. If no such CSP exists, then $\OPT(g)\subsetneq\OPT(f)$ for all optimal~$f$. An example of WCSP~$g$ for which no such CSP exists is in Figure~\ref{fig:3cycle}.

It is natural to ask which WCSPs possess this property. 
Though we are currently unable to provide a full characterization of such WCSPs, we identify two such classes:

\begin{theorem}[\cite{schlesinger1976,Werner-PAMI07}]\label{th:case_csp_exists}
If the LP relaxation~\eqref{eq:LPrelax} of a WCSP $g\in\R^T$ is tight, then $\OPT(g)=\SOL(A)$ for some $A\subseteq T$.
\end{theorem}
\begin{proof}
If the LP relaxation~\eqref{eq:LPrelax} is tight, then there exists a vector $f \in \R^T$ such that $\UB(f)=\max_{x\in \domain^\vars}\VAL gx$ and $f$ is a reparametrization of $g$, i.e., $\VAL fx=\VAL gx$ for all $x\in \domain^\vars$, thus, $f$ is also optimal for~\eqref{eq:LP}. It follows that the sets of optimal assignments for~$f$ and~$g$ coincide. By Theorem~\ref{th:f_optimal}, $A^*(f)$~is the required CSP.
\end{proof}

\begin{theorem}
If a WCSP $g\in\R^T$ has a unique optimal assignment (i.e., $\ourabs{\OPT(g)}=1$), then $\OPT(g)=\SOL(A)$ for some $A\subseteq T$.
\end{theorem}
\begin{proof}
The case with~$|\domain|=1$ is trivial, so let $|\domain|\geq2$ and $\OPT(g)=\{x\}$.

We claim that~$A=\ourset{(\scope,x[\scope])\mid \scope\in\constraints}$ is the required CSP, i.e., $\SOL(A)=\{x\}$. For contradiction, suppose that $x'\in\SOL(\CSP)$ and $x'\neq x$. By~$x'\in\SOL(\CSP)$, we necessarily have that~$x'[\scope]=x[\scope]$ for all~$\scope\in\constraints$. By definition~\eqref{eq:energy'}, this implies $\VAL g{x'}=\VAL gx$. Thus, $\{x',x\}\subseteq\OPT(g)$, which is contradictory with~${\ourabs{\OPT(g)}=1}$.
\end{proof}

\subsection{Properties of General Super-Reparametrizations}\label{se:arbitrary_suprepar}

Finally, we present one property of general super-reparametrizations~$f$ of a fixed WCSP $g\in\R^T$, i.e., $f$~is only feasible (but possibly not optimal) for~\eqref{eq:LP}. 

\begin{theorem}\label{th:t_non_opt}
For every $g\in \R^T$ we have
\begin{equation}
\ourset{A^*(f) \mid \text{$f\in\R^T$ is a super-reparametrization of~$g$}}
= \ourset{A^*(f)\mid f\in \mathbb{R}^T}.
\label{eq:t_non_opt}
\end{equation}
\end{theorem}
\begin{proof}
The inclusion $\subseteq$ is trivial. To prove $\supseteq$, let~$f'\in\mathbb{R}^T$ be arbitrary. Define~$f\in\R^T$ as $f_t = B(g)/|\constraints|+\iverson{t\in A^*(f')}$, $t\in T$. Clearly, $f$ is a super-reparametrization of~$g$ due to $\VAL fx \geq B(g)\geq \VAL gx$ for any~$x\in\domain^\vars$. In addition, by definition of~$f$, $\max_{t\in \TPA{\scope}}f_t = B(g)/|\constraints|+1$ for any $\scope\in\constraints$, hence $A^*(f')=A^*(f)$.
\end{proof}

Theorem~\ref{th:t_non_opt} shows that the left-hand set in~\eqref{eq:t_non_opt} does not depend on~$g$ at all. Therefore, if we approximately optimize~\eqref{eq:LP}, i.e., we find a (possibly non-optimal) super-reparametrization~$f$ of~$g$, then there is in general no relation between sets $\OPT(f)$ and $\OPT(g)$. However, as shown in~\S\ref{sec:super-reparam}, an arbitrary super-reparametrization still maintains the valuable property that it provides an upper bound~$B(f)$ on the optimal value $\max_x\VAL gx$ of WCSP~$g$.

\section{Hardness Remarks}\label{se:hardness}

Unsurprisingly, a number of decision and optimization problems related to~\eqref{eq:LP} is computationally hard since the optimization problem~\eqref{eq:LP} is hard itself. We overview a number of such problems here.

\begin{theorem}\label{th:NPh_superrepar}
The following problem is NP-complete: Given~$f,g\in\mathbb{Q}^T$, decide whether $f$~is \emph{not\/} a super-reparametrization of~$g$ (i.e., whether $f$~is \emph{not\/} feasible for~\eqref{eq:LP}).
\end{theorem}
\begin{proof}
Membership in NP can be shown easily by the notion of a non-deterministic algorithm~\cite[\S10]{alsuwaiyel1999algorithms}. First, one can choose any~$x\in\domain^\vars$ and then in polynomial time decide whether~$\VAL fx<\VAL gx$. 

To show NP-hardness, we perform a reduction from CSP satisfiability which is known to be NP-complete. Let $A\subseteq T$~be a CSP. We would like to decide whether $\SOL(A)\neq\emptyset$.

Let us define~$g\in\{0,1\}^T$ by
\begin{equation}
    g_\scope(k) = \iverson{(\scope,k)\in A} \quad \forall (\scope,k)\in T.
\end{equation}
Thus, for any~$x\in\domain^\vars$, $\VAL gx$ equals to the number of constraints in CSP~$A$ that are satisfied by the assignment $x$. So, $\VAL gx\in\{0, 1, \ldots, |\constraints|\}$ and $\VAL gx=|\constraints|$ if and only if $x\in\SOL(A)$. Consequently, $\max_x \VAL gx\leq|\constraints|-1$ if and only if $\SOL(A)=\emptyset$.

We define~$f\in\mathbb{Q}^T$ by $f_t = (|\constraints|-1)/|\constraints|$, $t\in T$. In analogy to Theorem~\ref{th:optimality_cond}, $\VAL fx=|\constraints|-1$ for all assignments~$x\in\domain^\vars$. Hence, $\SOL(A)=\emptyset$ if and only if $f$~is a super-reparametrization of~$g$.
\end{proof}

\begin{corollary}\label{co:NPh_cone}
The following problem is NP-complete: Given~$d\in\mathbb{Q}^T$, decide whether $d\notin M^*$.
\end{corollary}
\begin{proof}
Membership in NP is analogous to Theorem~\ref{th:NPh_superrepar}. The question of whether $f$~is not a super-reparametrization of~$g$ from Theorem~\ref{th:NPh_superrepar} reduces to whether $d=f-g\notin M^*$.
\end{proof}

\begin{corollary}\label{co:NPh_optimal}
The following problem is NP-complete: Given~$f,g\in\{0,1\}^T$ where $f$~is a super-reparametrization of $g$, decide whether $f$~is optimal for~\eqref{eq:LP}.
\end{corollary}
\begin{proof}
Membership in NP follows from Theorem~\ref{th:optimality_cond}: as in Theorem~\ref{th:NPh_superrepar}, one can choose~$x\in\domain^\vars$ and then in polynomial time decide whether~$\VAL fx=\VAL gx$ and~$x\in\SOL(A^*(f))$.

The hardness part is completely analogous to the proof of Theorem~\ref{th:NPh_superrepar} except that we define~$f\in\{0,1\}^T$ by $f_t = 1$, $t\in T$, so $\UB(f)=|\constraints|$. Clearly, $f$~is element-wise greater than or equal to~$g$, so it is a super-reparametrization. Moreover, $|\constraints|=\max_x \VAL gx$ if and only if $\SOL(A)\neq\emptyset$, so $f$~is optimal for~\eqref{eq:LP} if and only if $\SOL(A)\neq\emptyset$.
\end{proof}

Recall that in formula~\eqref{eq:deact_exists}, the number $\delta$~had the concrete value given by Theorem~\ref{th:deact_exists}. However, sometimes the value of~$\delta$ can be decreased while \eqref{eq:deact_exists}~still remains to be an $R$-deactivating direction for~$\CSP$. Finding a small such~$\delta$ is desirable because then \eqref{eq:deact_exists} results in smaller objective values $\VAL dx$, as explained in~\S\ref{se:properties_of_the_method}. Unfortunately, finding the least value of~$\delta$ is likely intractable:

\begin{theorem}\label{th:NPh_deac}
The following problem is NP-complete: Given~$\delta \in \mathbb{Q}$ and $\deactivating\subseteq \CSP \subseteq T$ satisfying $\SOL(\CSP)=\SOL(\CSP-\deactivating)$, decide whether vector~$d$ given by~\eqref{eq:deact_exists} is \emph{not\/} an $\deactivating$-deactivating direction for~$\CSP$.
\end{theorem}

\begin{proof}
Membership in NP is analogous to Theorem~\ref{th:NPh_superrepar}. Since conditions (a) and (b) from Definition~\ref{de:t-deactivating} are satisfied, the question boils down to deciding whether $d\in M^*$. This cannot be reduced to the case in Corollary~\ref{co:NPh_cone} because $d$~in~\eqref{eq:deact_exists} has a special form.

To show hardness, we proceed by reduction from the 3-coloring problem~\cite{rosen2000handbook,karp1972reducibility}: given a graph~$G^*=(V^*,E^*)$, decide whether it is 3-colorable. Let~$G=(\vars,\constraints)$ be the graph sum (also known as the disjoint union of graphs) of~$G^*$ and~$K_4$~\cite[\S8.1.2]{rosen2000handbook}. $K_4$~is the complete graph with 4 vertices. Informally, $G$~is the graph obtained from~$G^*$ by adding 4 new vertices and including an edge between each pair of these new vertices.

Let CSP~$\CSP$ have the structure $(\domain,\vars,\constraints)$ where $|\domain|=3$ and
\begin{equation}\label{eq:NPh_graphs_A}
    \CSP = \{(\{i,j\},(k_i,k_j)) \mid \{i,j\}\in \constraints, k\in\domain^{\{i,j\}}, k_i\neq k_j\}.
\end{equation}
Hence, any $x\in\domain^\vars$ can be interpreted as an assignment of colors to the nodes of~$G$ and $x\in\SOL(\CSP)$ if and only if $x$~is a 3-coloring of~$G$. Since $G$~contains $K_4$~as its subgraph, it is not 3-colorable and $\CSP$~is unsatisfiable. Hence, setting~$\deactivating=\CSP$ satisfies $\SOL(\CSP-\deactivating)=\SOL(\emptyset)=\emptyset=\SOL(\CSP)$. 

For the purpose of our reduction, let us define $\delta = (|\constraints|-2)/2 >0$. We will show that for such a setting, $d$~is not an $\deactivating$-deactivating direction for~$\CSP$ if and only if $G^*$~is 3-colorable.

Plugging the above-defined sets~$\CSP$ and~$\deactivating$ into the definition of~$d$ in~\eqref{eq:deact_exists} yields
\begin{equation}\label{eq:NPh_energy_deac}
    \VAL dx= \sum_{\substack{\{i,j\}\in\constraints\\x_i\neq x_j}} (-1) + \sum_{\substack{\{i,j\}\in \constraints\\x_i= x_j}} \delta = \delta (|\constraints|-\COL(x))-\COL(x)
\end{equation}
where $\COL(x) = \ourabs{\ourset{\{i,j\}\in\constraints \mid x_i \neq x_j}}$ is the number of edges in~$G$ whose adjacent vertices have different colors in assignment~${x\in\domain^\vars}$.

If $G^*$~is 3-colorable, then there is~$x\in\domain^\vars$ such that~$\COL(x)=|\constraints|-1$. In other words, only for a single edge in $\constraints-E^*$ (i.e., edge of graph $K_4$), the adjacent vertices are assigned the same color, so $\VAL dx = -|\constraints|/2<0$ by~\eqref{eq:NPh_energy_deac} and definition of~$\delta$. Hence, $d$~is not an $\deactivating$-deactivating direction for~$\CSP$.

For the other case, if $G^*$~is not 3-colorable, then for any~$x\in\domain^\vars$, $\COL(x)\leq|\constraints|-2$. The reason is that for at least one edge in~$K_4$ and at least one edge in~$G^*$, the adjacent vertices will be assigned the same color in any assignment. By substituting the value of~$\delta$ and a simple manipulation of~\eqref{eq:NPh_energy_deac}, one obtains
\begin{equation}
    \VAL dx = \delta |\constraints|-(\delta +1) \COL(x) = |\constraints|\left(|\constraints|-2-\COL(x)\right)/2 \geq 0
\end{equation}
where the term in brackets is non-negative due to $\COL(x)\leq|\constraints|-2$ for any~$x\in\domain^\vars$. So, $d$~is an $\deactivating$-deactivating direction for~$\CSP$.
\end{proof}

In connection to~\S\ref{se:minimal_csp}, a number of decision problems concerning the minimal CSP have been also proved hard. For recent results, see~\cite{gottlob2012minimal,escamocher2018pushing}.

\section{Summary and Discussion}\label{se:summary}

We have proposed a method to compute upper bounds on the (maximization version of) WCSP. The WCSP is formulated as a linear program with an exponential number of constraints, whose feasible solutions are super-reparametrizations of the input WCSP instance (i.e., WCSP instances with the same structure and greater or equal objective values). Whenever the CSP formed by the active (i.e., maximal in their weight functions) tuples of a feasible WCSP instance is unsatisfiable, there exists an improving direction (in fact, a certificate of unsatisfiability of this CSP) for the linear program. As this approach provides only a subset of all possible improving directions, it can be seen as a local search. We showed how these improving directions can be generated by constraint propagation (or, more generally, by other methods to prove unsatisfiability of a CSP). We showed that super-reparametrizations are closely related to the dual cone to the well-known marginal polytope.

Special cases of our approach are the VAC / Augmenting DAG algorithm~\cite{Cooper-AI-2010,Koval76,Werner-PAMI07}, which uses arc consistency, and the algorithm in~\cite{komodakis2008beyond}, which uses cycle consistency. We have implemented the approach for singleton arc consistency, resulting in VSAC-SR algorithm. When compared to existing soft local consistency methods on a public dataset, VSAC-SR provides comparable or better bounds for many instances. Although the runtimes are higher than those of the simpler techniques, such as EDAC or VAC, one can control different trade-offs between bound quality and runtime by stopping the method prematurely, e.g., when the step size becomes small or terminating already with a greater value of~$\theta$ (see Footnote~\ref{foo:decreasing}).

The approach in general requires storing all the weights of the super-reparametrized WCSP instance. This may be a drawback when the domains are large and/or the weight functions are not given explicitly as a table of values but rather by an algorithm (oracle).

We expect our improved bounds to be useful when solving practical WCSP instances. Applications may include, e.g., using the method in preprocessing, pruning the search space during branch-and-bound search, providing tighter optimality gaps for solutions proposed by heuristic approaches, or generating high-quality proposals for solutions, as in~\cite{komodakis2008beyond}. However, we have done no experiments with this, so it is open whether the tighter bounds would outweigh the higher complexity of the algorithm. Due to the many options in which the method can be used, we leave this for future research. In addition, our approach can be also useful to solve more WCSP instances even without search (similarly, as the VAC algorithm solves all supermodular WCSPs without search) or, given a suitable primal heuristic, to solve WCSP instances approximately.

The approach can be straightforwardly extended to WCSPs with different domain sizes\footnote{In fact, our implementation already supports different domain sizes. We did not present our theoretical results for this generalized setting only to simplify notation.} and some weights equal to minus infinity (i.e., some constraints being hard). Of course, further experiments would be needed to evaluate the quality of the bounds if infinite weights are allowed. The WCSP framework also usually assumes a pre-defined specific finite bound that is updated during branch-and-bound~\cite{cooper2020valued} -- although the presented pseudocode does not support this, it is not difficult to extend it in this way.

Finally, we presented a theoretical analysis of the concept of super-reparametrizations of WCSPs, describing the properties of optimal super-reparametrizations and characterizing the set of active-tuple CSPs induced by different optimal super-reparametrizations. For example, even an optimal super-reparametrization may change the set of optimal assignments, as shown in~\S\ref{se:props_optimal}. Additionally, we have shown that general (i.e., possibly non-optimal) super-reparametrizations are only weakly related to the original WCSP instance.

\vspace{0.5cm}
\noindent \textbf{Acknowledgments:} Tom\'{a}\v{s} Dlask and Tom\'{a}\v{s} Werner were supported by the Czech Science Foundation (grant 19\nobreakdash-09967S) and the OP VVV project CZ.02.1.01/0.0/0.0/16\textunderscore019/0000765. Tom\'{a}\v{s} Dlask was also supported by the Grant Agency of the Czech Technical University in Prague (grants SGS19/170/OHK3/3T/13 and SGS22/061/OHK3/1T/13). Simon de Givry was supported by the French {\em Agence nationale de la Recherche}  (ANR\nobreakdash-19\nobreakdash-P3IA\nobreakdash-0004 ANITI). Open access publishing was supported by the National Technical Library in Prague. 

\begin{appendix}

\section{Example: EDAC, VAC, and VSAC}\label{ap:example}

\renewcommand\thefigure{\arabic{figure}}
\setcounter{figure}{6}    

In this section, we present an example that shows how EDAC, VAC, and VSAC can be gradually enforced in a WCSP. As in~\cite{Cooper-AI-2010,de2005existential}, we restrict this section to binary WCSPs and assume that~$\{i\}\in \constraints$ for each~$i\in \vars$. Recall that VAC and VSAC were formally defined already in Remark~\ref{re:virtual_consistencies}: a WCSP~$f\in\mathbb{R}^T$ is virtual $\Phi$-consistent if~$A^*(f)$ has a non-empty $\Phi$-consistency closure. Now, we proceed to define EDAC in our formalism. For this purpose, let~$E = \ourset{\scope\in\constraints \mid |\scope|=2}$ be the set of binary constraint scopes (as in~\S\ref{se:final_alg} and~Example~\ref{ex:cc}) and~$N_i=\{j\in \vars\mid \{i,j\}\in E\}$ be the set of neighbors of node~$i$ in the undirected graph~$(\vars,E)$.

\begin{definition}[{\cite{Cooper-AI-2010,de2005existential}}]
Let~$A\subseteq T$, $\{i,j\}\in E$, and~$k_i\in D$. The tuple~$(\{i\},k_i)$ is \emph{simply supported} by~$j$ in~$A$ if~$\exists k_j\in D$ such that~$(\{i,j\},(k_i,k_j))\in A$. The tuple~$(\{i\},k_i)$ is \emph{fully supported} by~$j$ in~$A$ if~$\exists k_j\in D$ such that~$(\{i,j\},(k_i,k_j))\in A$ and~$(\{j\},k_j)\in A$.
\end{definition}

\begin{definition}[{\cite{Cooper-AI-2010,de2005existential}}]\label{def:edac}
Let~$g\in\mathbb{R}^T$ be a binary WCSP and~$\preceq$ be a total order on its set of variables~$\vars$. WCSP~$g$ is \emph{Existential Directional Arc Consistent (EDAC)} w.r.t.~$\preceq$ if the following conditions hold:
\begin{itemize}
    \item $\forall i \in V\; \forall k\in D\; \forall j\in N_i$: $i \preceq j \implies$ $(\{i\},k)$ is fully supported by~$j$ in~$A^*(g)$,
    \item $\forall i \in V\; \forall k\in D\; \forall j\in N_i$: $j \preceq i \implies$ $(\{i\},k)$ is simply supported by~$j$ in~$A^*(g)$,
    \item $\forall i \in V\; \exists k\in D$: $(\{i\},k)\in A^*(g)$ and $\forall j\in N_i$: $(\{i\},k)$ is fully supported by~$j$ in~$A^*(g)$.
\end{itemize}
\end{definition}

\begin{remark}
The notions of Definition~\ref{def:edac} correspond to the notions in~\cite{Cooper-AI-2010,de2005existential} but they are tailored to our formalism. The first difference is that we do not consider infinite weights (i.e., hard constraints), which simplifies some conditions in the definition. The second difference is that the `baseline' for a weight~$g_t$, $t\in T_S$ is neither~${\perp}$ nor~0, but rather~$\max_{t'\in T_S}g_{t'}$. Consequently, we require $g_t=\max_{t'\in T_S}g_{t'}$ (i.e., $t\in A^*(g)$) in the definitions instead of~$g_t=0$.
\end{remark}

\begin{remark}
Let us also comment on the individual conditions in Definition~\ref{def:edac}. The first condition is known as Directional Arc Consistency w.r.t.~$\preceq$~\cite{Cooper-AI-2010,de2005existential}. The first and second condition together are known as Full Directional Arc Consistency w.r.t.~$\preceq$~\cite{Cooper-AI-2010,de2005existential}. Finally, the third condition is Existential Arc Consistency~\cite{Cooper-AI-2010,de2005existential}.
\end{remark}

\begin{figure}
\centering
\begin{subfigure}[t]{0.4077\textwidth}
\centering
\includegraphics[width=\textwidth]{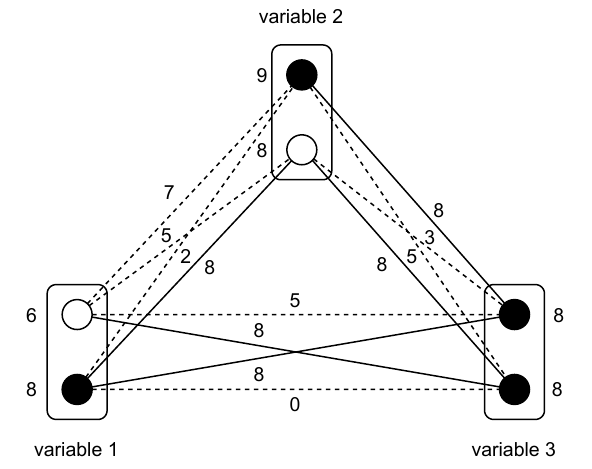}
\caption{WCSP~$f^1$ with~$\UB(f^1)=49$. This WCSP is not EDAC w.r.t.~$\leq$.} \label{fig:appendix_1}
\end{subfigure}\hspace{0.05\textwidth}
\begin{subfigure}[t]{0.4077\textwidth}
\centering
\includegraphics[width=\textwidth]{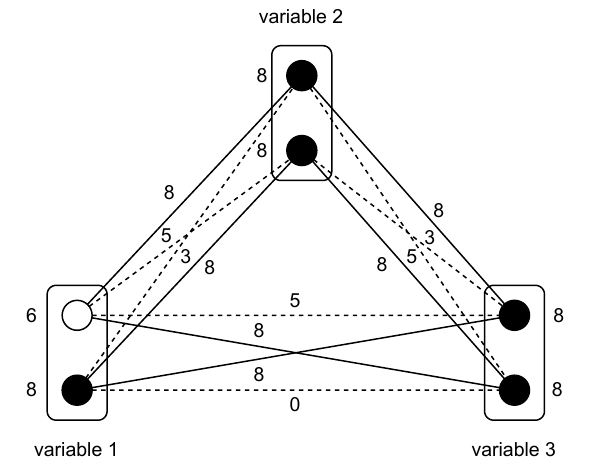}
\caption{WCSP~$f^2$ with~$\UB(f)=48$. This WCSP is EDAC w.r.t.~$\leq$ but not VAC.}
\label{fig:appendix_2}
\end{subfigure}
\begin{subfigure}[t]{0.4077\textwidth}
\centering
\includegraphics[width=\textwidth]{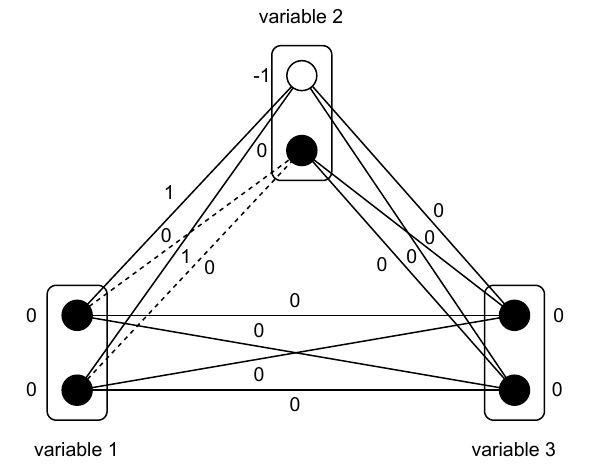}
\caption{WCSP~$f^2-f^1\in M^{\perp}$.}
\label{fig:appendix_3}
\end{subfigure}\hspace{0.05\textwidth}
\begin{subfigure}[t]{0.38804\textwidth}
\centering
\includegraphics[width=\textwidth]{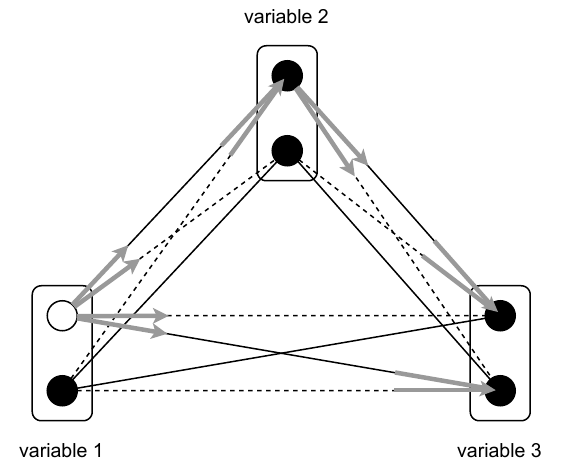}
\caption{Propagation of arc consistency in CSP~$A^*(f^2)$.}
\label{fig:appendix_4}
\end{subfigure}
\begin{subfigure}[t]{0.4077\textwidth}
\centering
\includegraphics[width=\textwidth]{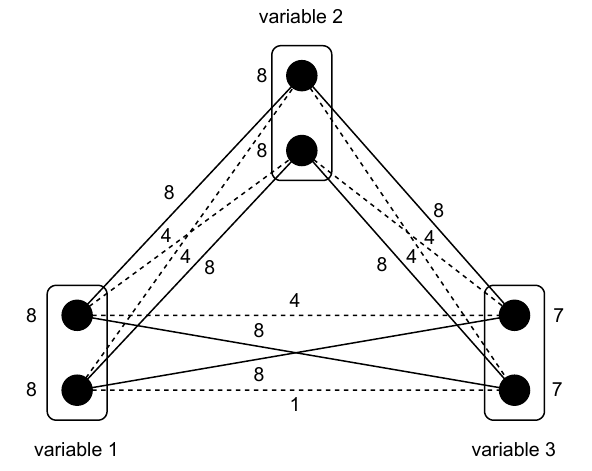}
\caption{WCSP~$f^3$ with~$\UB(f)=47$. This WCSP is VAC but not VSAC.}
\label{fig:appendix_5}
\end{subfigure}\hspace{0.05\textwidth}
\begin{subfigure}[t]{0.4077\textwidth}
\centering
\includegraphics[width=\textwidth]{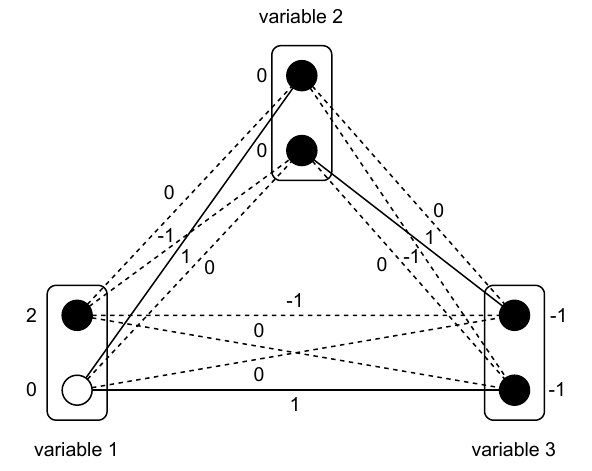}
\caption{WCSP~$f^3-f^2\in M^{\perp}$.}
\label{fig:appendix_6}
\end{subfigure}
\caption{Enforcing EDAC and VAC in a WCSP via reparametrizations.} 
\end{figure}

We now proceed to show our example where EDAC, VAC, and VSAC will be gradually enforced. The initial WCSP~$f^1$ is depicted in Figure~\ref{fig:appendix_1}. The structure of this WCSP is~$(\domain,\vars,\constraints)$ where~$\domain=\{\val{a,b}\}$, $\vars=\{1,2,3\}$, and~$\constraints=\{\{1\},\{2\},\{3\},\{1,2\},\{2,3\},\{1,3\}\}$. The names of the variables are indicated in the figures. To simplify the figures throughout this example, we do not state the names of the values in them -- the upper value is~$\val a$ and the lower value is~$\val b$. The optimal objective value of WCSP~$f^1$ is 43, which is attained, e.g., by the assignment~$x=(\val{a,a,a})$.

WCSP~$f^1$ is not EDAC (w.r.t. the natural ordering by~$\leq$) because the tuple~$(\{1\},\val b)$ is not fully supported by variable~2. To make WCSP~$f^1$ EDAC, it is sufficient to shift weight from the unary tuple~$(\{2\},\val a)$ to the binary weight function with scope~$\{1,2\}$, which results in WCSP~$f^2$ depicted in Figure~\ref{fig:appendix_2}. WCSP~$f^2$ is a reparametrization of~$f^1$ due to~$f^2-f^1\in M^{\perp}$ (depicted in Figure~\ref{fig:appendix_3}).

WCSP~$f^2$ is EDAC w.r.t.~$\leq$ but it is not VAC because the AC closure of~$A^*(f^2)$ is empty. To see that the AC closure is empty, we can follow the propagations that are depicted in Figure~\ref{fig:appendix_4}. The arrows point from the cause of forbidding a tuple to the newly forbidden tuple. First, we can forbid the tuple~$(\{1,2\},(\val{a,a}))$ because the tuple~$(\{1\},\val a)$ is forbidden. Second, we can forbid the tuple~$(\{2\},\val a)$ because both tuples~$(\{1,2\},(\val{a,a}))$ and~$(\{1,2\},(\val{b,a}))$ are forbidden. Next, we gradually forbid~$(\{2,3\},(\val{a,a}))$, $(\{3\},\val a)$, $(\{1,3\},(\val{a,b}))$, and~$(\{3\},\val b)$. This leads to domain wipe-out in variable~3. By shifting the weights against the direction of the arrows (as depicted in Figure~\ref{fig:appendix_4}), we make this WCSP VAC. This yields the WCSP~$f^3$ in Figure~\ref{fig:appendix_5} which is a reparametrization of~$f^2$. For clarity, we also show how the weights were transformed in Figure~\ref{fig:appendix_6}.

\begin{figure}
\centering
\begin{subfigure}[t]{0.38804\textwidth}
\centering
\includegraphics[width=\textwidth]{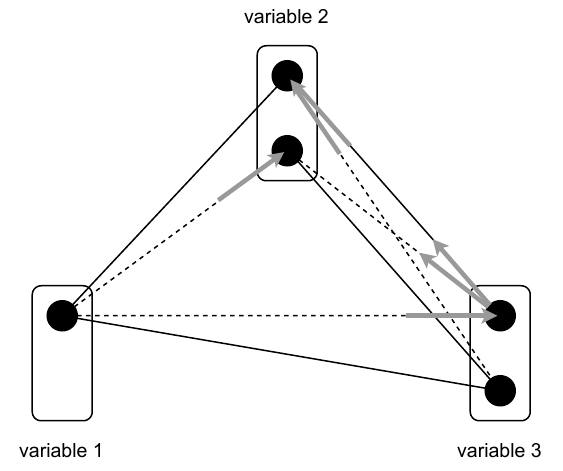}
\caption{Propagation of arc consistency in CSP~$A^*(f^3)\vert_{x_1=\val a}$.} \label{fig:appendix_7}
\end{subfigure}\hspace{0.05\textwidth}
\begin{subfigure}[t]{0.4077\textwidth}
\centering
\includegraphics[width=\textwidth]{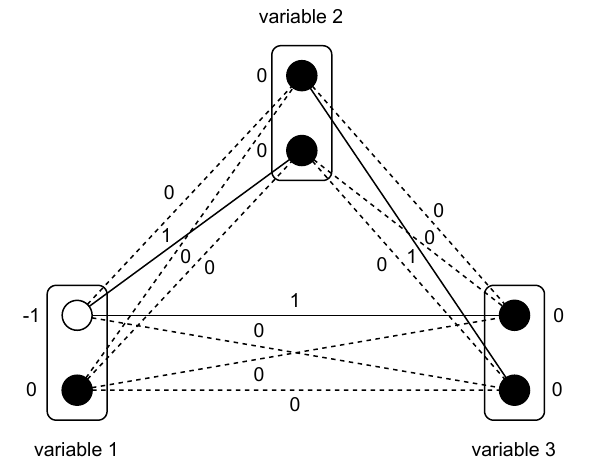}
\caption{$\{(\{1\},\val a)\}$-deactivating direction~$d$ for CSP~$A^*(f^3)$.}
\label{fig:appendix_8}
\end{subfigure}
\begin{subfigure}[t]{0.8154\textwidth}
\centering
\includegraphics[width=0.5\textwidth]{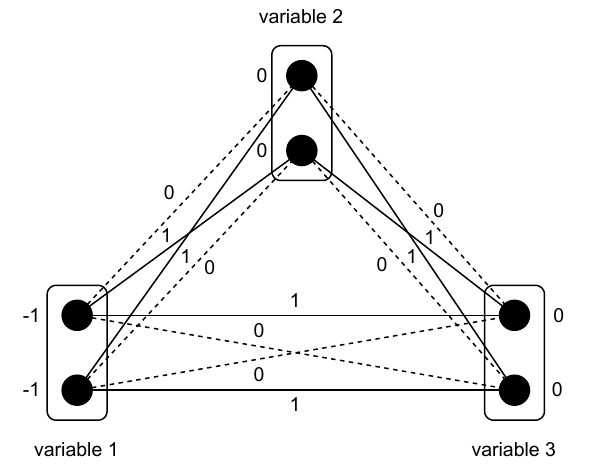}
\caption{$\{(\{1\},\val a), (\{1\},\val b)\}$-deactivating direction~$d''$ for CSP~$A^*(f^3)$.}
\label{fig:appendix_9}
\end{subfigure}

\begin{subfigure}[t]{0.8154\textwidth}
\centering
\includegraphics[width=0.5\textwidth]{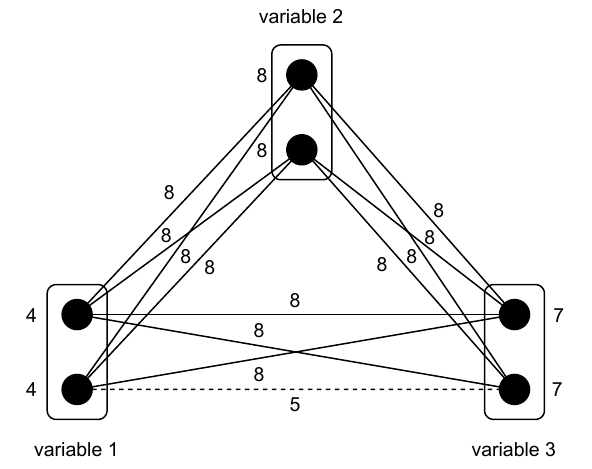}
\caption{WCSP~$f^4=f^3+4d''$ with~$\UB(f^4)=43$. This WCSP is VSAC.}
\label{fig:appendix_10}
\end{subfigure}
\caption{Enforcing VSAC via super-reparametrizations.} 
\end{figure}

WCSP~$f^3$ is VAC and even OSAC, so the bound~$\UB(f^3)$ cannot be improved by reparametrizations (without introducing a ternary weight function with scope~$\{1,2,3\}$). WCSP~$f^3$ is however not VSAC because it has empty SAC closure. Thus, $A^*(f^3)$~is unsatisfiable and we are able to construct a super-reparametrization of WCSP~$f^3$ with a better bound (recall Theorem~\ref{th:CSPcert} and~\S\ref{se:constraint_propagation_scheme}). We show next the details of the construction, following our results from~\S\ref{se:improving_directions},~\S\ref{se:line_search}, and~\S\ref{se:final_alg}.

Figure~\ref{fig:appendix_7} shows how arc consistency is enforced in the CSP~$A^*(f^3)\vert_{x_1=\val a}$ and the notation is analogous to Figure~\ref{fig:appendix_4}. In detail, we gradually forbid tuples~$(\{3\},\val a)$, $(\{2,3\},(\val a,\val a))$, $(\{2\},\val a)$, and finally~$(\{2\},\val b)$, which leads to domain wipe-out in variable~2. Consequently, the AC closure of~$A^*(f^3)\vert_{x_1=\val a}$ is empty. To derive this, it suffices that the tuples $P=\{(\{1,3\},(\val a,\val a)), (\{2,3\},(\val a,\val b)), (\{1,2\},(\val a,\val b))\}$ are forbidden in~$A^*(f^3)$. Applying Theorem~\ref{th:deact_exists} with~$A=T-P\supseteq A^*(f^3)$ and~$\deactivating=\{(\{1\},\val a)\}$ results in $\{(\{1\},\val a)\}$-deactivating direction~$d$ for~$A^*(f^3)$ that is shown in Figure~\ref{fig:appendix_8}. Analogously, we can compute a $\{(\{1\},\val b)\}$-deactivating direction~$d'$ for~$A^*(f^3)$ (not shown). By summing these deactivating directions together (i.e., using Theorem~\ref{th:combining_deactivating_vectors} which in this case yields~$\delta=1$), we obtain a $\{(\{1\},\val a),(\{1\},\val b)\}$-deactivating direction~$d''=d+d'$ for~$A^*(f^3)$ that certifies unsatisfiability of~$A^*(f^3)$. By Theorem~\ref{th:improve_bound}, we compute the step size~$\alpha=\min\{\beta,\gamma\}=4$ and obtain WCSP~$f^4=f^3+\alpha d''$ which is shown in Figure~\ref{fig:appendix_10} and is VSAC.

\end{appendix}

\clearpage

\bibliographystyle{abbrv}
\bibliography{mybibliography.bib}

\end{document}